\documentclass[preprint,10pt]{elsarticle}

\usepackage{arydshln}
\usepackage{mathrsfs}
\usepackage{amsmath}
\usepackage{amsfonts}
\usepackage{amssymb}
\usepackage[all,cmtip]{xy}
\usepackage[colorlinks]{hyperref}
\usepackage{xcolor}
\usepackage{cleveref}
\usepackage{ulem}
\usepackage[abs]{overpic}
\usepackage{float}
\usepackage{enumitem}
\usepackage{tikz}
\usetikzlibrary{decorations.pathmorphing}

\usepackage{etoolbox}
\patchcmd{\thebibliography}
  {\settowidth}
  {\setlength{\itemsep}{0pt plus 0.1pt}\settowidth}
  {}{}
\apptocmd{\thebibliography}
  {\small}
  {}{}

\definecolor{DarkRed}{RGB}{200,20,20}
\definecolor{SkyBlue}{rgb}{0.16, 0.32, 0.75}

\usepackage[all,cmtip]{xy}
\usepackage{tikz-cd}
\usepackage[toc,page]{appendix}
\usepackage{graphicx}
\usepackage{subcaption}
\usepackage[utf8]{inputenc}
\usepackage[english]{babel}
\usepackage{dsfont}
\usepackage{xcolor}
\usepackage{arydshln}
\usepackage{faktor}
\usepackage{mathtools}
\usepackage{scalerel}
\usepackage{booktabs}

\def\endpf{\hbox{\vrule height1.5ex width.5em}}

\newcommand\xdashmapsto[2][]{\mathrel{\mapstochar\xdashrightarrow[#1]{#2}}}

\makeatletter
\newcommand*{\da@rightarrow}{\mathchar"0\hexnumber@\symAMSa 4B }
\newcommand*{\da@leftarrow}{\mathchar"0\hexnumber@\symAMSa 4C }
\newcommand*{\xdashrightarrow}[2][]{%
	\mathrel{%
		\mathpalette{\da@xarrow{#1}{#2}{}\da@rightarrow{\,}{}}{}%
	}%
}
\newcommand{\xdashleftarrow}[2][]{%
	\mathrel{%
		\mathpalette{\da@xarrow{#1}{#2}\da@leftarrow{}{}{\,}}{}%
	}%
}
\newcommand{\xdashdownarrow}[2][]{%
	\mathrel{%
		\mathpalette{\da@xarrow{#1}{#2}\da@downarrow{}{}{\,}}{}%
	}%
}
\newcommand*{\da@xarrow}[7]{%
	\sbox0{$\ifx#7\scriptstyle\scriptscriptstyle\else\scriptstyle\fi#5#1#6\m@th$}%
	\sbox2{$\ifx#7\scriptstyle\scriptscriptstyle\else\scriptstyle\fi#5#2#6\m@th$}%
	\sbox4{$#7\dabar@\m@th$}%
	\dimen@=\wd0 %
	\ifdim\wd2 >\dimen@
	\dimen@=\wd2 %   
	\fi
	\count@=2 %
	\def\da@bars{\dabar@\dabar@}%
	\@whiledim\count@\wd4<\dimen@\do{%
		\advance\count@\@ne
		\expandafter\def\expandafter\da@bars\expandafter{%
			\da@bars
			\dabar@ 
		}%
	}%  
	\mathrel{#3}%
	\mathrel{%   
		\mathop{\da@bars}\limits
		\ifx\\#1\\%
		\else
		_{\copy0}%
		\fi
		\ifx\\#2\\%
		\else
		^{\copy2}%
		\fi
	}%   
	\mathrel{#4}%
}
\makeatother

\usepackage{graphicx}

\newcommand\tikzmark[1]{%
  \tikz[remember picture,overlay]\coordinate (#1);}

\newcommand{\underbracedmatrixll}[2]{%
  \left(\;\hspace{-.27in}
  \smash[b]{\underbrace{
    \begin{matrix}#1\end{matrix}
  }_{#2}}
  \;\right.
  \vphantom{\underbrace{\begin{matrix}#1\end{matrix}}_{#2}}
}

\newcommand{\underbracedmatrixrr}[2]{%
  \left. \;
  \smash[b]{\underbrace{
    \begin{matrix}#1\end{matrix}
  }_{#2}}
  \;\hspace{-.32in}\right)
  \vphantom{\underbrace{\begin{matrix}#1\end{matrix}}_{#2}}
}

\newtheorem{theorem}{Theorem}[section]
\newtheorem{lemma}[theorem]{Lemma}

\newtheorem{proposition}[theorem]{Proposition}

\newtheorem{example}[theorem]{Example}
\newtheorem{remark}[theorem]{Remark}

\begin{document}

\begin{frontmatter}
\title{Canonical blow-ups of Lagrangian and Orthogonal Grassmannians}   
%}

\author[aff1]{Hanlong Fang} % Keep the funding comment inside \thanks if needed

\ead{hlfang@pku.edu.cn}

\author[aff2]{Alex Massarenti}
\ead{msslxa@unife.it}

\author[aff1]{Xian Wu}
\ead{xianwu.ag@gmail.com}

% Correct: Use \address instead of \affiliation
\address[aff1]{School of Mathematical Sciences, Peking University, Beijing 100871, China}
\address[aff2]{Dipartimento di Matematica e Informatica, Universit\` a di Ferrara, Via Machiavelli 30, 44121 Ferrara, Italy}

% Corresponding author\cortext[cor1]{Corresponding author}

% Emails with explicit linking

\begin{abstract}
Let $\mathbf{LG}(V\oplus V^*)$ and $\mathbf{OG}^+(V\oplus V^*)$ denote the Lagrangian and orthogonal Grassmannians endowed with the natural $\mathbb{G}_m$-actions, respectively. Thaddeus proved that over $\mathbb{C}$, the Hilbert quotients $\mathbf{LG}(V\oplus V^*)\!/\!/\mathbb{G}_m$ and $\mathbf{OG}^+(V\oplus V^*)\!/\!/\mathbb{G}_m$ are isomorphic to the wonderful compactifications of the spaces of symmetric and skew‑symmetric matrices of maximal ranks, that is, the spaces of complete quadrics and complete skew-forms, respectively. In this paper, we construct the universal families of these Hilbert quotients by explicitly blowing up the corresponding isotropic Grassmannians, resulting in smooth toroidal compactifications of the spaces of symmetric and skew‑symmetric matrices of maximal ranks (before projectivization), which have simple normal crossing boundary divisors and include the spaces of the complete bilinear forms among these divisors. Specifically, we prove that over any algebraically closed field, the universal families of these Hilbert quotients are smooth, and the Hilbert quotients themselves are isomorphic to the spaces of the complete bilinear forms. Over an algebraically closed field of characteristic zero, we prove that the universal families are weak Fano varieties with vanishing higher cohomology groups for their tangent bundles, and are therefore locally rigid. Furthermore, we show that these universal families naturally resolve the Landsberg-Manivel rational maps from projective spaces to isotropic Grassmannians. 
\end{abstract}

\begin{keyword}
Wonderful compactifications \sep
Hilbert quotients \sep
Landsberg-Manivel maps  \sep 
Spherical varieties \sep
blow-up

%% PACS codes here, in the form: \PACS code \sep code

%% MSC codes here, in the form: \MSC code \sep code
%% or \MSC[2008] code \sep code (2000 is the default)

\end{keyword}

\end{frontmatter}

% manifold}
%by a  submanifold of real codimension two}

%\medskip

%\centerline{ Dedicated to the memory of  Professor Qi-keng Lu }%
%\medskip
%\centerline{ Dedicated to the memory of Professor Qi-Keng Lu}

%\vspace{3cm} 

%\tableofcontents

\section{Introduction}
The study of the moduli spaces of complete bilinear forms originated from classical 19th-century works in enumerative and projective geometry \cite{Cha,Hir,Schu}, and later evolved towards a more geometric perspective in higher dimensions, establishing connections to fields such as combinatorics and algebraic statistics, with key contributions from \cite{Van,Se2,Ty,Va2,La1,DP,DP2,Th1,LLT,MMM}, among others. In particular, De Concini and Procesi developed the foundational theory of what later became known as wonderful compactifications, generalizing earlier constructions to all symmetric spaces of adjoint type. More recently, such spaces have been used in constructing other moduli spaces  \cite{Th1,LoH,CM,CCM,ORCW}. In particular, Thaddeus proved that over $\mathbb C$, the Hilbert quotients of the natural one-dimensional torus actions on Grassmannians, Lagrangian Grassmannians, and orthogonal Grassmannians are isomorphic to the spaces of complete collineations, complete quadrics, and complete skew-forms, respectively.  

The theory of wonderful compactifications constitutes a cornerstone in the study of homogeneous varieties with various applications in arithmetic geometry and representation theory \cite{Fa,Sp1,Lu2,DKKS}. Nevertheless, the current framework fails for  homogeneous spaces whose automorphism groups possess a non-trivial center.  The development of a more general framework is still an enticing open question, while a major step forward has been taken by Kausz \cite{Ka} in 2000, who constructed a novel modular compactification for the general linear group. A simple uniform picture was recently established in \cite{FW} to incorporate the Kausz compactifications and the spaces of complete collineations. In particular, it was shown that over any algebraically closed field, the Kausz-type compactifications are the universal families for the Hilbert quotients of the 
$\mathbb G_m$-action on Grassmannians, which aligns with Thaddeus's results.

Investigating birational maps from projective spaces to smooth prime Fano varieties has been significantly advanced by the study of special Cremona transformations with a smooth irreducible base scheme \cite{Zak,CK89,ESB89,HSS92,Ru09,AS16,FH}. Landsberg and Manivel \cite{LM} constructed natural birational maps from projective spaces onto minuscule varieties, for which the base scheme may be singular. Zak \cite{Zak}, and Fu and Hwang \cite{FH} factored the rational maps associated with complex quadrics into a blow-up and a blow-down. Ding \cite{Di} further generalized to any compact Hermitian symmetric spaces, based on the blow-up centers constructed via the Bia{\l}ynicki-Birula decomposition (first appeared in \cite{Fang}) and the theory of varieties of minimal rational tangents \cite{HM}. It was shown in \cite{FW} that the Landsberg--Manivel maps from projective spaces onto Grassmannians can naturally factor through the Kausz-type compactifications  over ${\rm Spec\,}\mathbb Z$.  

%$\phi:\mathbb P^r\dashrightarrow Y$

%A birational map $\phi: X \dashrightarrow Y$ is said to admit a weak factorization if it can be decomposed into a sequence of blow-ups and blow-downs with smooth centers, and admit a strong factorization if, furthermore, all blow-ups in the sequence occur before any blow-downs. While weak factorization has been proven for complete nonsingular varieties in characteristic zero in \cite{AKMW}, the question of whether a strong factorization always exists is an open problem. A celebrated theorem by Abramovich-Karu-Matsuki-W\l odarczyk \cite{AKMW}shows that a birational map between complete nonsingular algebraic varieties X1 and X2 over an algebraically closed field K of characteristic zero

The focus of this paper is to study the universal families of the two remaining types of Hilbert quotients that arise in Thaddeus's work.
We explicitly construct the universal families by generalizing the blow‑up technique of \cite{FW}. Using the concrete description, we are able to examine their differential and birational geometric properties. %These constructions also yield new examples of homogeneous spaces with a non‑trivial center in their automorphism groups that admit smooth toroidal compactifications analogous to Kausz’s, which are constructed from Grassmannians with one‑dimensional torus actions (see \cite{FZ23} for examples with multi‑dimensional torus actions).

%Similar ideas have been developed for Faltings-Lafforgue spaces \cite{FZ23} and KSBA moduli of marked cubic surfaces \cite{FSW}.
\smallskip

%Let $G(p,n)$, $0<p<n$, be the Grassmannian  consisting of $p$-planes in the  $n$-space. 

Let us be more precise. Let $V$ be a free $\mathbb{Z}$-module of rank $n$. The module $V \oplus V^*$ possesses a canonical symplectic form $\omega$ defined by $\omega((v_1, \alpha_1), (v_2, \alpha_2)):=\alpha_2(v_1) - \alpha_1(v_2)$. The {\it Lagrangian Grassmannian} $\mathbf{LG}(V \oplus V^*)$ is a smooth, projective scheme over $\mathbb{Z}$ which represents the functor assigning to each commutative ring $R$ the set 
\begin{equation*}
%\mathbf{LG}(V \oplus V^*)(R) = 
\left\{ L \subset (V \oplus V^*)\otimes_{\mathbb{Z}} R \left|\, \begin{array}{l} L \text{ is a} \text{ rank } n \text{ direct summand of } (V \oplus V^*)\otimes_{\mathbb{Z}} R, \\ \omega_R(u, w) = 0 \text{ for all } u, w \in L \end{array} \right. \right\}.
\end{equation*}
%where $\omega_R$ is the $R$-linear extension of $\omega$. 
%The Lagrangian Grassmannian is a closed subscheme of the Grassmannian scheme $\mathbf{Gr}(n,V\oplus V^*)$, which parameterizes all rank $n$ direct summands of $V\oplus V^*$.$\mathbf{Gr}(n,V\oplus V^*)$
The Pl\"ucker embedding of the Grassmannian induces an embedding
\begin{equation}\label{lgpll}
e_{{\bf LG}}:\mathbf{LG}(V\oplus V^*)\xhookrightarrow{\,\,\,\,\,\,\,\,\,\,\,} \mathbb P\big(\bigwedge\nolimits^n(V\oplus V^*)\big).
\end{equation}

The $\mathbb G_m$-action on $V\oplus V^*$ defined by \begin{equation}\label{ltl}
t\cdot(v,\alpha)\mapsto(v,t\alpha),\ \ 
\,\,(v,\alpha)\in V\oplus V^*, 
\end{equation} induces a $\mathbb G_m$-action on $\mathbf{LG}(V \oplus V^*)$. Define a subgroup ${\rm GL}(V^*)\hookrightarrow {\rm Sp}(V\oplus V^*,\omega)$ by
\begin{equation*}
\begin{split}
A&\longmapsto \left(\begin{matrix}
    A^{-T}&0\\0&A
\end{matrix}\right)
\end{split}.   
\end{equation*}
Noticing that ${\rm GL}(V^*)\times\mathbb G_m$ stabilizes each summand in the decomposition $\bigwedge\nolimits^{n}(V\oplus V^*)=\bigoplus\limits_{k=0}^n\bigwedge\nolimits^{k}V\otimes\bigwedge\nolimits^{n-k}V^*,$
we can define a ${\rm GL}(V^*)\times\mathbb G_m$-equivariant rational map
\begin{equation}\label{kspel}
\mathcal {KL}(V):\mathbf{LG}(V\oplus V^*)\xdashrightarrow{\,\,\,\,\,\,\,\,\,\,\,\,\,\,\,}\mathbb P\big(\bigwedge\nolimits^n(V\oplus V^*)\big)\times\prod\limits_{k=0}^n\mathbb P\big(\bigwedge\nolimits^{k}V\otimes\bigwedge\nolimits^{n-k}V^*\big).  
\end{equation}
Denote by $\mathcal {TL}(V)$ the closure of the image of $\mathcal {KL}(V)$, and $\mathcal {ML}(V)$ the image of $\mathcal {TL}(V)$ under the natural projection
\begin{equation}\label{nprojl}
\mathbb P\big(\bigwedge\nolimits^n(V\oplus V^*)\big)\times\prod\limits_{k=0}^n\mathbb P\big(\bigwedge\nolimits^{k}V\otimes\bigwedge\nolimits^{n-k}V^*\big)\longrightarrow \prod\limits_{k=0}^n\mathbb P\big(\bigwedge\nolimits^{k}V\otimes\bigwedge\nolimits^{n-k}V^*\big).
\end{equation}

%Next, we turn to the case of Orthogonal Grassmannians.
Similarly, the {\it orthogonal Grassmannian} $\mathbf{OG}(V \oplus V^*)$ is a smooth, projective scheme over $\mathbb{Z}$ which represents the functor assigning to each commutative ring $R$ the set 
\begin{equation*}
%\mathbf{LG}(V \oplus V^*)(R) = 
\left\{ L \subset (V \oplus V^*)\otimes_{\mathbb{Z}} R \left|\, \begin{array}{l} L \text{ is a} \text{ rank } n \text{ direct summand of } (V \oplus V^*)\otimes_{\mathbb{Z}} R, \\ Q_R|_L=0 \end{array} \right. \right\}.
\end{equation*}
%where $Q_R$ is the $R$-linear extension of $Q$.Similarly, the Orthogonal Grassmannian is a closed subscheme of $\mathbf{Gr}(n,V\oplus V^*)$. 
Here $V\cong\mathbb Z^n$ with $n\geq 2$, and $V \oplus V^*$ possesses a a canonical quadratic form $Q$ defined by $Q(v, \alpha): = \alpha(v)$. %for all $(v_1, \alpha_1),(v_2, \alpha_2)\in V \oplus V^*$. 
We denote by $\mathbf{OG}^+(V\oplus V^*)$ the connected component of $\mathbf{OG}(V\oplus V^*)$ that contains the maximal isotropic submodule $V \oplus \{0\}\subset V\oplus V^*$.

The first canonical embedding of $\mathbf{OG}^+(V\oplus V^*)$ can be given as follows (see \cite{Knus} for details). Let $C(V\oplus V^*, Q)$ be the Clifford algebra with a $\mathbb Z$-linear map 
$i:V\oplus V^*\rightarrow C(V\oplus V^*, Q)$ such that $[i(x)]^2 = Q(x)$. The spinor representation $\rho:C(V\oplus V^*,Q)\rightarrow {\rm End}_{\mathbb Z}(\bigwedge V^*)$  defined by $\rho(v,\alpha)=l_{\alpha}+\iota_{v}$ splits into half-spinor representations $\bigwedge^{\rm even}V^*$, $\bigwedge^{\rm odd}V^*$, where 
$l_{\alpha}$ is the left multiplication
and $i_{v}$ is the derivation.
A spinor $\phi\in \bigwedge^{\rm even}V^*$ is pure if its annihilator $L_{\phi}:=\{x\in V\oplus V^*:x\cdot\phi=0\}$ is a maximal isotropic submodule. Each $L\in\mathbf{OG}^+(V\oplus V^*)$ corresponds to a unique pure spinor line 
$[\phi_L]$. we have the spinor embedding
\begin{equation}\label{eog}
e_{{\bf OG^+}}:\mathbf{OG}^+(V \oplus V^*)\xhookrightarrow{\,\,\,\,\,\,\,\,\,\,\,}\mathbb{P}\big(\bigwedge\nolimits^{\rm even}V^*\big).
\end{equation}

Similarly, we can define a ${\rm GL}(V^*)\times\mathbb G_m$-equivariant rational map
\begin{equation}\label{kspeo}
\mathcal {KO}(V):\mathbf{OG}^+(V\oplus V^*)\xdashrightarrow{\,\,\,\,\,\,\,\,\,\,\,\,\,\,\,}\mathbb P\big(\bigoplus\limits_{k=0}^{[n/2]}\bigwedge\nolimits^{2k}V^*\big)\times\prod\limits_{k=0}^{[n/2]}\mathbb P\big(\bigwedge\nolimits^{2k}V^*\big). 
\end{equation}
Denote by $\mathcal {TO}(V)$ the closure of the image of $\mathcal {KO}(V)$, and $\mathcal {MO}(V)$ the image of $\mathcal {TO}(V)$ under the natural projection
\begin{equation}\label{nprojo}
\mathbb P\big(\bigoplus\limits_{k=0}^{[n/2]}\bigwedge\nolimits^{2k}V^*\big)\times\prod\limits_{k=0}^{[n/2]}\mathbb P\big(\bigwedge\nolimits^{2k}V^*\big)\longrightarrow \prod\limits_{k=0}^{[n/2]}\mathbb P\big(\bigwedge\nolimits^{2k}V^*\big).
\end{equation}
When there is no ambiguity, we may write $\mathbf{LG}(V \oplus V^*)$, $\mathcal {TL}(V)$, $\mathcal {ML}(V)$, $\mathbf{OG}^+(V \oplus V^*)$, $\mathcal {TO}(V)$, $\mathcal {MO}(V)$ as $\mathrm{LG}(n,2n)$, $\mathcal {TL}_{n}$, $\mathcal {ML}_{n}$, $\mathrm{OG}^+(n,2n)$, $\mathcal {TO}_{n}$, $\mathcal {MO}_{n}$, respectively.

We begin by proving that the above  constructions yield smooth compactifications of homogeneous spaces with favorable properties.
\begin{theorem}\label{gwondl}
The inverse of (\ref{kspel}) extends to a regular morphism \begin{equation}\label{rspnl}
RL_{n}:\mathcal {TL}(V)\xrightarrow{\,\,\,\,\,\,\,\,\,\,\,\,} \mathbf{LG}(V \oplus V^*).    
\end{equation} 
$\mathcal {TL}(V)$ is smooth and projective over ${\rm Spec}\,\mathbb Z$ with a ${\rm GL}(V^*)\times\mathbb G_m$-action. There are $2n$ ${\rm GL}(V^*)\times\mathbb G_m$-stable smooth prime divisors with simple normal crossingss $D^+_1, 
\cdots,D^+_n,D^-_1$, $\cdots,D^-_n$ such that the following holds.
\begin{enumerate}[label={\rm(\Alph*)}]
\item $D^+_1\cong D^-_1\cong\mathcal {ML}(V)$ are smooth and projective over ${\rm Spec}\,\mathbb Z$. 

\item There is a ${\rm GL}(V^*)\times\mathbb G_m$-equivariant flat retraction 
\begin{equation*}
\mathcal {PL}_{n}:\mathcal {TL}(V)\xrightarrow{\,\,\,\,\,\,\,\,\,\,\,\,}D_1^-\cong\mathcal {ML}(V)  %\ \ ({resp.}\,  \mathcal P_{s,p,n}:\mathcal T_{s,p,n}\longrightarrow D_1^-) 
\end{equation*} such that the restriction  $\mathcal {PL}_{n}|_{D^+_1}:D_1^+\rightarrow D_1^-$ %(resp. $\mathcal P_{s,p,n}|_{D^+_1}$) 
is an isomorphism and that for $2\leq i\leq n$,
\begin{equation*}
\mathcal {PL}_{n}(D^-_i)=\mathcal {PL}_{n}(D^+_{n+2-i})=:\check D_i.\,\,    
\end{equation*}

\item After base change to an algebraically closed field of characteristic not $2$, $\mathcal {TL}(V)$ is a toroidal ${\rm GL}(V^*)$-spherical variety. The complement of the open
${\rm GL}(V^*)$-orbit consists of $D^+_1, 
\cdots,D^+_n,D^-_1$, $\cdots,D^-_n$. The closures of the ${\rm GL}(V^*)$-orbits are one-to-one given by
\begin{equation}\label{inrulel}
\bigcap\nolimits_{\,\,i\in I^+}D^+_i\mathbin{\scaleobj{1.1}{\bigcap}} \bigcap\nolimits_{\,\,i\in I^-}D^-_i%\mathbin{\scaleobj{2.1}{\backslash}} (\bigcup\nolimits_{j\notin I^-}D^-_j\mathbin{\scaleobj{1}{\bigcup}}\,\bigcup\nolimits_{j\notin I^+}D^+_j),
\end{equation} 
for subsets $I^+,I^-\subset\{1,2,\cdots,n\}$
such that $\min(I^+)+\min(I^-)\geq n+2$ with the convention $\min(\emptyset)=+\infty$. $D_1^-$ is wonderful with the ${\rm GL}(V^*)$-stable divisors $\check D_i$,  $2\leq i\leq n$.
\end{enumerate}
\end{theorem}
%Notice that in Property (C), one cannot drop the condition that $\mathbb K$ is algebraically closed or of characteristic not $2$, because the normal forms of symmetric matrices under congruence are not determined solely by rank in these cases.

\begin{theorem}\label{gwondo}
The inverse of (\ref{kspeo}) extends to a regular morphism \begin{equation}\label{rspno}
RO_{n}:\mathcal {TO}(V)\xrightarrow{\,\,\,\,\,\,\,\,\,\,\,\,} \mathbf{OG}^+(V \oplus V^*).    
\end{equation} 
$\mathcal {TO}(V)$ is smooth and projective over ${\rm Spec}\,\mathbb Z$ with a ${\rm GL}(V^*)\times\mathbb G_m$-action. There are $2[\frac{n}{2}]$ ${\rm GL}(V^*)\times\mathbb G_m$-stable smooth prime divisors with simple normal crossingss $D^+_1, 
\cdots,D^+_{[\frac{n}{2}]},D^-_1$, $\cdots,D^-_{[\frac{n}{2}]}$ such that the following holds.
\begin{enumerate}[label={\rm(\Alph*)}]
\item $D^+_1\cong D^-_1\cong\mathcal {MO}(V)$ are smooth and projective over ${\rm Spec}\,\mathbb Z$. 

\item There is a ${\rm GL}(V^*)\times\mathbb G_m$-equivariant flat retraction 
\begin{equation*}
\mathcal {PO}_{n}:\mathcal {TO}(V)\xrightarrow{\,\,\,\,\,\,\,\,\,\,\,\,} D_1^-\cong\mathcal {MO}(V)  %\ \ ({resp.}\,  \mathcal P_{s,p,n}:\mathcal T_{s,p,n}\longrightarrow D_1^-) 
\end{equation*} such that the restriction  $\mathcal {PO}_{n}|_{D^+_1}:D_1^+\rightarrow D_1^-$ %(resp. $\mathcal P_{s,p,n}|_{D^+_1}$) 
is an isomorphism and that for $2\leq i\leq [\frac{n}{2}]$,
\begin{equation*}
\mathcal {PO}_{n}(D^-_i)=\mathcal {PO}_{n}(D^+_{[\frac{n}{2}]+2-i})=:\check D_i\,\,    
\end{equation*}

\item After base change to an algebraically closed field of characteristic not $2$, $\mathcal {TO}(V)$ is a toroidal ${\rm GL}(V^*)$-spherical variety. The complement of the open
${\rm GL}(V^*)$-orbit consists of $D^+_1, 
\cdots,D^+_{[\frac{n}{2}]},D^-_1$, $\cdots,D^-_{[\frac{n}{2}]}$. The closures of ${\rm GL}(V^*)$-orbits are one-to-one given by
\begin{equation*}%\label{inruleo}
\bigcap\nolimits_{\,\,i\in I^+}D^+_i\mathbin{\scaleobj{1.1}{\bigcap}} \bigcap\nolimits_{\,\,i\in I^-}D^-_i%\mathbin{\scaleobj{2.1}{\backslash}} (\bigcup\nolimits_{j\notin I^-}D^-_j\mathbin{\scaleobj{1}{\bigcup}}\,\bigcup\nolimits_{j\notin I^+}D^+_j),
\end{equation*} 
for subsets $I^+,I^-\subset\{1,2,\cdots,[\frac{n}{2}]\}$
such that $\min(I^+)+\min(I^-)\geq [\frac{n}{2}]+2$ with the convention $\min(\emptyset)=+\infty$. $D_1^-$ is wonderful with the ${\rm GL}(V^*)$-stable divisors $\check D_i$,  $2\leq i\leq [\frac{n}{2}]$.
\end{enumerate}
\end{theorem}

Recall from \cite{Ty,Va2,La1} that the space of complete quadrics is the closure of the graph of 
\begin{equation*}
\begin{split}
\mathbb P\big({\rm Sym}^2V^*\big)&\xdashrightarrow{\,\,\,\,\,\,\,\,\,\,\,}
\mathbb P\big({\rm Sym}^2\bigwedge\nolimits^{0}V^*\big)\times\mathbb P\big({\rm Sym}^2\bigwedge\nolimits^{1}V^*\big)\times\cdots\times\mathbb P\big({\rm Sym}^2\bigwedge\nolimits^{n}V^*\big),  
\end{split}    
\end{equation*} 
and the space of complete skew-forms is that of
\begin{equation*}
\begin{split}
\mathbb P\big(\bigwedge\nolimits^2V^*\big)&\xdashrightarrow{\,\,\,\,\,\,\,\,\,\,\,}
\mathbb P\big(\bigwedge\nolimits^2\bigwedge\nolimits^{0}V^*\big)\times\mathbb P\big(\bigwedge\nolimits^2\bigwedge\nolimits^{1}V^*\big)\times\cdots\times\mathbb P\big(\bigwedge\nolimits^2\bigwedge\nolimits^{[n/2]}V^*\big).   
\end{split}    
\end{equation*}
The spaces of complete bilinear forms are naturally realized as boundary divisors as follows.

\begin{proposition}\label{red2l} 
\begin{enumerate}[label={\rm(\Alph*)}]
    \item The space of complete quadrics in $\mathbb P(V)$ is isomorphic to $\mathcal {ML}(V)$.    

    \item The space of complete skew-forms on $V$ is isomorphic to $\mathcal {MO}(V)$.  
\end{enumerate}
\end{proposition}

For a projective variety $X$ with an algebraic group action $G$, there is an open set $U \subset X$ such that the orbit closures of points in $U$ form a flat family of subschemes of $X$. Hence, there is a morphism from $U$ to ${\rm Hilb}(X)$, the Hilbert scheme of $X$. Following \cite{BS,Kap}, the closure of the image of $U$ in ${\rm Hilb}(X)$ is called the {\it Hilbert quotient} and denoted by $X\! /\!/G$. %Over the field of complex numbers, Thaddeus \cite{Th1} proved that $\mathcal {ML}(V)$ and $\mathcal {MO}(V)$ are isomorphic to the Hilbert quotients of the $\mathbb G_m$-action induced by (\ref{ltl}) on $\mathbf{LG}(V \oplus V^*)\! /\!/ \mathbb G_m$ and $\mathbf{OG}^+(V \oplus V^*)\! /\!/ \mathbb G_m$, respectively,  by . 

We now state one of our main results in this paper.
\begin{theorem}\label{modulil} 
After base change to an algebraically closed field, the following holds.
\begin{enumerate}[label={\rm(\Alph*)}]
    \item The Hilbert quotient $\mathbf{LG}(V \oplus V^*)\! /\!/\mathbb G_m$ is isomorphic to $\mathcal {ML}(V)$; its universal family is given by $\mathcal {PL}_{n}:\mathcal {TL}(V)\longrightarrow\mathcal {ML}(V)$. 

    \item The Hilbert quotient $\mathbf{OG}^+(V \oplus V^*)\! /\!/\mathbb G_m$ is isomorphic to $\mathcal {MO}(V)$; its universal family is given by $\mathcal {PO}_{n}:\mathcal {TO}(V)\longrightarrow\mathcal {MO}(V)$.
    
\end{enumerate}
Consequently, all three Hilbert quotients in \cite{Th1} have smooth universal families.
\end{theorem}

%Next, we turn to the birational geometry of $\mathcal{TL}_n$ and $\mathcal {TO}_n$.

To investigate the recognition problem,  Landsberg--Manivel \cite{LM03,LM,LM01,LM02,LM04} established new relations between the representation theory of complex simple Lie groups and the projective algebraic geometry of their homogeneous varieties. %In particular, they introduced algorithms that construct new varieties from old, which lead to new proofs of the classification of compact Hermitian symmetric spaces and the Cartan-Killing classification of complex simple Lie algebras. 
Their construction yields explicit rational maps, which generalized the maps Zak \cite{Zak} introduced for classifying Severi varieties.
Denote by $\mathcal {LML}:\mathbb P^{\frac{n(n+1)}{2}}\dashrightarrow \mathrm{LG}(n,2n)$ and $\mathcal {LMO}:\mathbb P^{\frac{n(n-1)}{2}}\dashrightarrow \mathrm{OG}^+(n,2n)$ the Landsberg--Manivel birational maps (see \S \ref{landmani} for details).   We establish the following natural factorizations.
\begin{theorem}\label{g4kausz}
\begin{enumerate}[label={\rm(\Alph*)}]
    \item The natural projection  $$\mathbb P\big(\bigwedge\nolimits^n(V\oplus V^*)\big)\xdashrightarrow{\,\,\,\,\,\,\,\,\,\,\,\,\,\,\,}\mathbb P\big((\bigwedge\nolimits^{n}V\otimes\bigwedge\nolimits^{0}V^*)\oplus(\bigwedge\nolimits^{n-1}V\otimes\bigwedge\nolimits^{1}V^*)\big)$$ induces a morphism ${\mathcal {KAL}}:\mathcal {TL}_{n}\rightarrow\mathbb P^{\frac{n(n+1)}{2}}$. We have the following commutative diagram.
\vspace{-0.05in}
\begin{equation*}%\label{K=F3l}
\begin{tikzcd}
&&\,\,\mathcal {TL}_{n}\arrow{d}\arrow{ddl}[swap]{\mathcal {KAL}}\arrow{ddr}{RL_{n}}\,\,&\\
&&\,\, {tl}_{n}\arrow{dl}[swap, pos=0.2]{{kal}}\arrow{dr}[pos=0.3]{{rl}_{n}}\,\,&\\
&\mathbb P^{\frac{n(n+1)}{2}}\arrow[dashed]{rr}{\mathcal {LML}}&& {\mathrm{LG}(n,2n)} \\
\end{tikzcd}\vspace{-20pt}\,.
\end{equation*}
Here $\mathcal {KAL}$ and $RL_n$ both consist of $2(n-1)$ successive blow-ups along smooth irreducible centers;  $kal$ and $rl_n$ both consist of $n-1$ successive blow-ups along smooth irreducible centers, each disjoint from the open subscheme over which $\mathcal {LML}$ is an isomorphism. 

\item The natural projection $$\mathbb P\big(\bigoplus\limits_{k=0}^{[n/2]}\bigwedge\nolimits^{2k}V^*\big)\xdashrightarrow{\,\,\,\,\,\,\,\,\,\,\,\,\,\,\,}\mathbb P\big(\bigwedge\nolimits^{0}V^*\oplus\bigwedge\nolimits^{2}V^*\big)$$ induces a morphism ${\mathcal {KAO}}:\mathcal {TO}_{n}\rightarrow\mathbb P^{\frac{n(n-1)}{2}}$. We have the following commutative diagram.
\vspace{-0.05in}
\begin{equation*}%\label{K=F3l}
\begin{tikzcd}
&&\,\,\mathcal {TO}_{n}\arrow{d}\arrow{ddl}[swap]{\mathcal {KAO}}\arrow{ddr}{RO_{n}}\,\,&\\
&&\,\, {to}_{n}\arrow{dl}[swap, pos=0.2]{{kao}}\arrow{dr}[pos=0.25]{{ro}_{n}}\,\,&\\
&\mathbb P^{\frac{n(n-1)}{2}}\arrow[dashed]{rr}{\mathcal {LMO}}&& {\mathrm{OG}^+(n,2n)} \\
\end{tikzcd}\vspace{-20pt}\,.
\end{equation*}
Here $\mathcal {KAO}$ and $RO_n$ both consist of  $n-2$ successive blow-ups along smooth irreducible centers;  $kao$ and $ro_n$ both consist of $[n/2]-1$ successive blow-ups along smooth irreducible centers, each disjoint from the open subscheme over which $\mathcal {LMO}$ is an isomorphism.

\end{enumerate}
\end{theorem}

%We state one of the main results of the paper as follows.

%Recall from \cite{Fu,BS,Kap} the definition of the Hilbert quotients as follows. 

%In the following, we fix an arbitrary algebraically closed field $\mathbb K$.

In the following, we %base change $\mathcal {TL}(V)$ and $\mathcal {TO}(V)$ to
work over a fixed algebraically closed field $\mathbb K$ of characteristic $0$. 

For a spherical variety, the nef cone is generated by its colors if it is simple, whereas the general situation is highly intricate (see \cite{Per}). Therefore, it is noteworthy that the nef cones here are cones over chain polytopes in the sense of Stanley \cite{Sta} (see Remark \ref{chp}).
\begin{theorem}[=Propositions \ref{nefcl}, \ref{nefco}]\label{pen}
Index the boundary divisors 
 as in \S \ref{foliation}, and define
\begin{equation}\label{dell}
\begin{split}
&\Delta^+_k:=-\sum\nolimits_{i=1}^k(k+1-i)D^+_i,\,\,\,\,\,\,\Delta^-_k:=-\sum\nolimits_{i=1}^k(k+1-i)D^-_i,\,\,0\leq k\leq n-1.
\end{split}
\end{equation} 
The nef cone
$\operatorname{Nef}(\mathcal {TL}_n)$ is generated by \begin{equation}\label{nefll}
    \{(RL_n)^*\mathcal O_{\mathrm{LG}(n,2n)}(1)+\Delta_i^++\Delta_{j}^-\,|\,0\leq i,j\leq n-1\,\,{\rm and}\,\,i+j\leq n \},
\end{equation}
and $\operatorname{Nef}(\mathcal {TO}_n)$  is generated by 
\begin{equation}\label{nefo}
\{(RO_n)^*\mathcal O_{\mathrm{OG}^+(n,2n)}(1)+\Delta_i^++\Delta_{j}^-\,|\,0\leq i\leq [\frac{n}{2}]-1,\,\,0\leq j\leq [\frac{n-1}{2}],\,\,{\rm and}\,\,i+j\leq [\frac{n}{2}] \}.    
\end{equation}
\end{theorem}
To our knowledge, this provides the first example in arbitrary dimension of a spherical variety with multiple closed orbits whose nef cone admits such a simple characterization.

%Cox rings were first introduced by Cox for toric varieties \cite{Cox}, and then his construction was generalized to projective varieties in \cite{HK}. Recall that Cox rings are the direct sum of the spaces of sections of all isomorphism classes of line bundles. Theoretically there is a way of computing the cone of movable divisors and the Mori chamber decomposition of a Mori dream space starting from an explicit presentation of its  can be summarized as follows. \begin{proposition}\label{prop:Cox-TLn}$\operatorname{Cox}(\mathcal{TL}_n)$ (resp.  $\operatorname{Cox}(\mathcal{TO}_n)$) is finitely generated by the canonical sections of the boundary divisors and colors of $\mathcal {TL}_n$ (resp. $\mathcal {TO}_n$). The extremal rays of $\operatorname{Mov}(\mathcal{TL}_n)$  (resp. $\operatorname{Mov}(\mathcal{TO}_n)$) are generated by effective divisor classes lying on facets of $\operatorname{Eff}(\mathcal{TL}_n)$ (resp. $\operatorname{Eff}(\mathcal{TO}_n)$) cut out by the colors. In particular, $\operatorname{Nef}(\mathcal {TL}_2)=\operatorname{Mov}(\mathcal {TL}_2)$, and for $n=3,4$, $\operatorname{Nef}(\mathcal{TL}_{n}) \subsetneq \operatorname{Mov}(\mathcal{TL}_{n})$; the Mori chamber decomposition of $\operatorname{Eff}(\mathcal {TL}_3)$ consists of $8$ chambers.
%\[\operatorname{Nef}(T L_n)\subsetneq \operatorname{Mov}(T L_n),\qquad n=3,4.\]\end{proposition}

A smooth variety is called {\it Fano} if its anticanonical bundle is ample, and {\it weak Fano} if its anticanonical bundle is nef and big. The spaces of complete collineations, complete quadrics, and complete skew-forms are all Fano varieties, as established in \cite{DP,LoH,Mas1,Mas2}. By enumerating all invariant curves, we obtain the (semi-)positivity of their universal families as follows.
\begin{theorem}\label{fano}
$\mathcal{TL}_n$ and $\mathcal{TO}_n$ are smooth weak Fano varieties. Among them, the Fano varieties are $\mathcal{TL}_1\cong\mathbb P^1$, $\mathcal{TL}_2$, $\mathcal{TO}_2\cong\mathbb P^1$, $\mathcal{TO}_3$, $\mathcal{TO}_4$, and $\mathcal{TO}_5$.
\end{theorem}
%(See \cite{Fang} for the corresponding results for the case of Grassmannians.)

It is well-known that local deformations of weak Fano varieties are unobstructed. Bien-Brion \cite{BB} proved that any smooth projective toroidal spherical Fano variety is locally rigid, meaning it admits no non-trivial local deformation. By exploiting the toroidal property and the 
explicit generating sets of the effective cones determined in Propositions \ref{effl}, \ref{effo}, 
%and discuss Cox rings, movable cones and Mori chamber decompositions (see Propositions \ref{prop:Cox-TLn1}, \ref{thm:Mov-TLn}, and Examples \ref{n=2}, \ref{n=3}, \ref{n=4}). 
we establish %the following vanishing results and thereby determine the local deformations of the weak Fano varieties $\mathcal{TL}_n$ and $\mathcal{TO}_n$.
\begin{theorem}\label{bignef}
For $X=\mathcal {TL}_n,\mathcal {TO}_n$, we have that
\begin{equation*}
 H^k(X, T_{X})=0   \,\,{\rm for\,\,all\,\,} k>0,
\end{equation*} 
where $T_X$ is the tangent bundle of $X$. In particular, $\mathcal {TL}_n$, $\mathcal {TO}_n$ are locally rigid.
\end{theorem}
To our knowledge, this yields very few known examples of spherical varieties exhibiting such vanishing properties with weak positivity on the anticanonical bundle (see \cite{FZ21} as well).

%\begin{theorem} The Kontsevich moduli space $\overline{M}_{0,0}(\mathbb P^2,2)$, parametrizing degree two stable maps from a nodal rational curve to $\mathbb P^2$, is isomorphic to the space of complete conics $Q(2)$ \cite[Section 0.4]{FP}.     Kontsevich moduli spaces Kontsevich moduli spaces are denoted by $\overline M_{g,n}(X,\beta)$ where $X$ is a projective scheme and $\beta\in H^2(X,\mathbb Z)$ is the homology class of a curve in $X$. A point in $\overline M_{g,n}(X,\beta)$ corresponds to a holomorphic map $\alpha$ from an $n$-pointed genus $g$ curve $C$ to $X$ such that $\alpha_*([C])=\beta$. When $X$ is a projective space or a Grassmannians the class $\beta$ is completely determined by its degree, similarly when $X$ is the product of two projective spaces we identify the class $\beta$ with its the bidegree. \end{theorem} 

\medskip

We now briefly describe the basic ideas for the proof. The construction is based on blowing up the relevant Grassmannians along specific coordinate planes in the ambient projective spaces of the Grassmannians' first canonical embeddings, which are determined by weight spaces under the torus action. We analyze their geometry via special coordinate charts in analogy with the canonical forms for congruence of matrices.  By observing a canonical basis, we determine the extremal rays of the nef cones. We establish the vanishing of the higher cohomology groups of the tangent bundles via the (semi-)positivity of the anticanonical bundles, combined with the explicit structure of the effective cones. We also use foundational results of Brion on spherical varieties (see, for instance, \cite{BPa,Br3,Br5}). %The, movable cones, Cox rings, and the connection 

%In special cases they are isomorphic to the ones used by Kausz, but we feel they are more general and convenient.
%we derive a natural flat map   $\mathcal P_{s,p,n}:\mathcal T_{s,p,n}\rightarrow\mathcal M_{s,p,n}$.  The fundamental work of Brion  on spherical varieties is another main technical tool used extensively in this paper. We determine the cone of effective divisors/curves of $\mathcal T_{s,p,n}$ and $\mathcal M_{s,p,n}$. Combined with the fibration structures, the holomorphic automorphism groups  follow. The invariants in Brion's theory can be computed conveniently in the  Mille Cr\^epes coordinates. By Kleiman's ampleness criterion, the (semi-)positivity of the anti-canonical bundles is a consequence of the calculation of the intersection numbers.   The organization of the paper is as follows. 

The organization of the paper is as follows. In \S \ref{iterated}, we fix some notations. In  \S \ref{vander}, we introduce the Mille Cr\^epes coordinate charts to parametrize $\mathcal {TL}_{n}$ and $\mathcal {TO}_{n}$. In \S \ref{foliationl}, we introduce the explicit Bia{\l}ynicki-Birula decomposition, and prove Proposition \ref{red2l}, and Theorems \ref{gwondl}, 
\ref{gwondo}, \ref{modulil}. In \S \ref{landmani}, we investigate the Landsberg--Manivel birational maps by proving Theorem \ref{g4kausz}. In \S \ref{kcpt}, we determine the effective and nef cones, compute the intersection numbers of the anticanonical divisors with curves in the Mori cone, and thereby prove  Theorems \ref{pen}, \ref{fano}, and \ref{bignef}.

\section{Preliminaries}\label{iterated}

We first introduce certain notation following \cite{FW}. %Denote by $\mathbb A^{m}$ the scheme ${\rm Spec}\,\mathbb Z\left[x_{1},\cdots,x_{m}\right]$, and by $\mathbb P^{m}$ the projective space ${\rm Proj}\,\mathbb Z\left[x_{0},\cdots,x_{m}\right]$. 
Define an index set
\begin{equation*}
\mathbb I_{n}:=\{(i_1,i_2,\cdots,i_n)\in\mathbb Z^n:1\leq i_1<i_2<\cdots<i_n\leq 2n\}.
\end{equation*}
For each $I=(i_1,i_2,\cdots,i_n)\in\mathbb I_{n}$, we define {\it Pl\"ucker coordinate functions} $P_{I}$ on $$\mathbb A^{n(2n)}:={\rm Spec}\mathbb Z\left[x_{ij}(1\leq i\leq n,1\leq j\leq 2n)\right]$$ as the subdeterminant of $(x_{ij})$ consisting of the $i_1^{th},\cdots,i_n^{th}$ columns.
For each point $x$ in the Grassmannian $G(n,2n)$, we denote by $\widetilde{x}$ an element in the preimage
\begin{equation*}%\label{gtor}
\pi^{-1}(x)\subset\mathcal G(n,2n):=\left\{\mathfrak p\in\mathbb A^{n(2n)}:P_{I}\notin\mathfrak p\,\,{\rm for\,\,a\,\,certain\,\,}I\in\mathbb I_{n}\right\}, 
\end{equation*} 
where  $\pi:\mathcal G(n,2n)\rightarrow G(n,2n)$ is the natural projection.

Choose a basis $\{e_1,\cdots,e_n\}$ for $V$ and let $\{e_1^*,\cdots,e_n^*\}$ denote its dual basis in $V^*$. 

For any integer $l$ with $0\leq l\leq n$, %, we may identify $\mathbf{Gr}(n,V\oplus V^*)$ with $G(n,2n)$. 
take the isotropic submodule $L_{12\cdots l}\in\mathbf{LG}(V \oplus V^*)$ generated by $e_{l+1},\cdots,e_n,e^*_{1},\cdots,e^*_{l}$. 
Viewed as a subscheme of $G(n, 2n)$, the Lagrangian Grassmannian $\mathbf{LG}(V \oplus V^*)$ has the following affine coordinate chart $U_{12\cdots l}\cong\mathbb A^{\frac{n(n+1)}{2}}$ around $L_{12\cdots l}$. 
\begin{equation*}%\label{ull}
U_{12\cdots l}:=\{\,\,\,\,\underbracedmatrixll{\,\,\,Z\\\,\,\,-X^T}{\,\,\,\,\,\,\,\,\,\,\,\,\,\,\,\,l\,\,\rm columns}
  \hspace{-.4in}\begin{matrix}
  &\hfill\tikzmark{a}\\
  &\hfill\tikzmark{b}  
  \end{matrix} \,\,\,\,\,
  \begin{matrix}
  0\\
I_{(n-l)\times(n-l)}\\
\end{matrix}\hspace{-.11in}
\begin{matrix}
  &\hfill\tikzmark{c}\\
  &\hfill\tikzmark{d}
  \end{matrix}\hspace{-.11in}\begin{matrix}
  &\hfill\tikzmark{g}\\
  &\hfill\tikzmark{h}
  \end{matrix}\,\,\,\,
\begin{matrix}
I_{l\times l}\\
0\\
\end{matrix}\hspace{-.11in}
\begin{matrix}
  &\hfill\tikzmark{e}\\
  &\hfill\tikzmark{f}\end{matrix}\hspace{-.3in}\underbracedmatrixrr{X\\W}{(n-l)\,\,\rm columns}\,\,\,\,\}
  \tikz[remember picture,overlay]   \draw[dashed,dash pattern={on 4pt off 2pt}] ([xshift=0.5\tabcolsep,yshift=7pt]a.north) -- ([xshift=0.5\tabcolsep,yshift=-2pt]b.south);\tikz[remember picture,overlay]   \draw[dashed,dash pattern={on 4pt off 2pt}] ([xshift=0.5\tabcolsep,yshift=7pt]c.north) -- ([xshift=0.5\tabcolsep,yshift=-2pt]d.south);\tikz[remember picture,overlay]   \draw[dashed,dash pattern={on 4pt off 2pt}] ([xshift=0.5\tabcolsep,yshift=7pt]e.north) -- ([xshift=0.5\tabcolsep,yshift=-2pt]f.south);\tikz[remember picture,overlay]   \draw[dashed,dash pattern={on 4pt off 2pt}] ([xshift=0.5\tabcolsep,yshift=7pt]g.north) -- ([xshift=0.5\tabcolsep,yshift=-2pt]h.south);
\end{equation*}
with coordinates
\begin{equation}\label{ulxl}
\begin{split}
&Z:=(\cdots,z_{ij},\cdots)_{1\leq i,j\leq l},\,\,X:=(\cdots,x_{ij},\cdots)_{1\leq i\leq l,l+1\leq j\leq n},\,\,W:=(\cdots,w_{ij},\cdots)_{l+1\leq i,j\leq n},
\end{split}   
\end{equation}
We adopt the convention $Z$ and $W$ are symmetric matrices with $z_{ij}=z_{ji}$ and $w_{ij}=w_{ji}$; $U_{0}$ is the affine coordinate chart for $l=0$. Similarly, for any $0\leq l\leq n$ and $i_1,\cdots,i_l,j_1,\cdots,j_{n-l}$ such that $\{i_1,\cdots,i_l,j_1,\cdots,j_{n-l}\}=\{1,2,\cdots,n\}$, we can construct an affine coordinate chart  $U_{i_1\cdots i_l}$ around $L_{12\cdots l}$, where $L_{i_1\cdots i_l}$ is  generated by $e_{j_1},\cdots,e_{j_{n-l}},e^*_{i_1},\cdots,e^*_{i_l}$. One may readily verify that all such coordinate charts cover $\mathbf{LG}(V \oplus V^*)$. 
%We remark that in the remainder of the paper  $U_l$, $0\leq l\leq r$, is always referred to (\ref{ull}).

We write  (\ref{lgpll}) as
$e_{\mathrm{LG}}:\mathrm{LG}(n,2n)\rightarrow\mathbb P^{N_{n}}$ where $\mathbb P^{N_{n}}$ has dimension $N_{n}:=\frac{(2n)!}{n!\cdot n!}-1$ and homogeneous coordinates $[\cdots ,z_I,\cdots]_{\,I\in\mathbb I_{n}}$. For $0\leq k\leq n$,
define index sets
\begin{equation*}%\label{ikk}
\mathbb I_{n}^{k}:=\big\{(i_1,\cdots,i_n)\in\mathbb Z^n:{ 1\leq i_1<\cdots<i_{n-k}\leq n\,;n+1\leq i_{n-k+1}<\cdots<i_n\leq 2n}\},
\end{equation*}
and linear subspaces
\begin{equation}\label{subl}
\{[\cdots ,z_I,\cdots]_{I\in\mathbb I_{n}}\in\mathbb {P}^{N_{n}}:z_I=0,\,\,\forall I\notin\mathbb I_{n}^k \}=:\mathbb {P}^{N^k_{n}}
\end{equation}
with dimension $|\mathbb I_{n}^{k}|-1$ and homogeneous coordinates $[\cdots ,z_I,\cdots]_{I\in\mathbb I^k_{n}}$.  Then (\ref{kspel}) takes the form 
\begin{equation}\label{fskl}
\begin{split}
\mathcal {KL}_{n}&=(e_{\mathrm{LG}}, f^0,\cdots,f^n):\mathrm{LG}(n,2n)\xdashrightarrow{\,\,\,\,\,\,\,\,\,\,\,\,\,\,\,}\mathbb {P}^{N_{n}}\times\mathbb {P}^{N^0_{n}}\times\cdots\times\mathbb {P}^{N^n_{n}}\\ 
&x\mapsto\big( [\cdots ,P_I(\widetilde x),\cdots]_{I\in\mathbb I_{n}},[\cdots,P_I(\widetilde x),\cdots]_{I\in\mathbb I^0_{n}},\cdots,[\cdots,P_I(\widetilde x),\cdots]_{I\in\mathbb I^n_{n}}\big),
\end{split}
\end{equation}
where $f^k:=F^k\circ e_{\mathrm{LG}}$ and $F^k:\mathbb {P}^{N_{n}}\dashrightarrow\mathbb {P}^{N^k_{n}}$ is the natural projection.

We give an alternative construction of $\mathcal {TL}_n$ as iterated blow-ups as follows. For each $0\leq k\leq n$, denote by $\mathcal S_k$ the ideal sheaf of $\mathcal O_{\mathrm{LG}(n,2n)}$ generated by $\{\left(e_{\mathrm{LG}}\right)^*z_{I}:{I\in \mathbb I_{n}^k}\}$ and by $S_k$ its vanishing locus. By the Laplace expansion of determinants, it is clear that
\begin{lemma}\label{sepal}
Fix integers $k,l$ with $0\leq k,l\leq n$. Consider the affine  coordinate chart $U_{12\cdots l}$ around $L_{12\cdots l}$. Then,  the restriction of $\mathcal S_k$ to $U_{12\cdots l}$ is generated by   
\begin{equation*}
   \left\{\begin{array}{ll}
    \{{\rm all\,\,}(l-k)\times(l-k)\,\,{\rm minors\,\,of\,\,}Z=(z_{ij})\},\,\,&k<l \\
    \{{\rm all\,\,}(k-l)\times(k-l)\,\,{\rm minors\,\,of\,\,}W=(w_{ij})\}, \,\,&k>l\\
       1,\,\,&k=l\\
   \end{array}\right.. 
\end{equation*}
\end{lemma}

%$S_k\subset G(p,n)$ by
%\begin{equation*}S_k:=\{ x\in {LG}(n,2n):\,P_{I}(\widetilde x)=0\,\,\forall{I\in \mathbb I_{n}^k}\}\,. \end{equation*}
For any permutation $\sigma$ of $\{0,1,\cdots,n\}$, let $g_0^{\sigma}:Y^{\sigma}_0\rightarrow \mathrm{LG}(n,2n)$ be the blow-up along $S_{\sigma(0)}$, and inductively define $g^{\sigma}_{i+1}$ as the blow-up along $(g^{\sigma}_0\circ g^{\sigma}_{1}\circ\cdots\circ g^{\sigma}_i)^{-1}(S_{\sigma(i+1)})$. All fit	into 
\vspace{-.08in}
\begin{equation}\label{sblow}
\small
\begin{tikzcd}
&Y^{\sigma}_n\ar{r}{g^{\sigma}_n}&Y^{\sigma}_{n-1}\ar{r}{g^{\sigma}_{n-1}}&\cdots\ar{r}{g^{\sigma}_1}&Y^{\sigma}_0\ar{r}{g^{\sigma}_0}&\mathrm{LG}(n,2n) \\
&&(g^{\sigma}_0\circ\cdots\circ g^{\sigma}_{n-1})^{-1}(S_{\sigma(n)})\ar[hook]{u}&\cdots&(g^{\sigma}_0)^{-1}(S_{\sigma(1)})\ar[hook]{u}&S_{\sigma(0)}\ar[hook]{u}\\
\end{tikzcd}.\vspace{-15pt}
\end{equation}
%It is clear that $Y_n^{\sigma}$ is independent of the choice of $\sigma$ by the following well-known result.
It is easy to verify that
\begin{lemma}\label{cpcl}
For any permutation $\sigma$, there is an isomorphism $\nu_{\sigma}:\mathcal {TL}_{n}\rightarrow Y^{\sigma}_n$ such that the following diagram commutes.
\vspace{-0.05in}
\begin{equation*}
\begin{tikzcd}
&\mathcal {TL}_n\arrow[dashed,swap]{rd}{\hspace{-0.03in}\mathcal ({KL}_{n})^{-1}} \arrow{rr}{\nu_{\sigma}}&&Y^{\sigma}_{n}\arrow{dl}{\hspace{-.03in}(g^{\sigma}_0\circ\cdots\circ g^{\sigma}_{n})} \\
&&\,\,\mathrm{LG}(n,2n)\,\,&\\
\end{tikzcd}\vspace{-20pt}\,.
\end{equation*}
In particular, there is a morphism $RL_{n}:\mathcal {TL}_{n}\rightarrow \mathrm{LG}(n,2n)$ extending $\mathcal ({KL}_{n})^{-1}$, which is also given by the projection of $\mathcal {TL}_{n}$ to the first factor $\mathbb {P}^{N_{n}}$ of the ambient space $\mathbb {P}^{N_{n}}\times\mathbb {P}^{N^0_{n}}\times\cdots\times\mathbb {P}^{N^n_{n}}$.
\end{lemma}

Note that the subgroup ${\rm GL}(V^*)\times\mathbb G_m\cong{\rm GL}_n\times\mathbb G_m$ has a natural action on $\mathrm{LG}(n,2n)$ by right matrix multiplication, which can be uniquely extended to $\mathcal {TL}_{n}$. 
\medskip

%In this subsection, we turn to Orthogonal Grassmannians. %Choose a basis $e_{1},\cdots,e_{n}$ for $V$ and its dual basis $e^*_{1},\cdots,e^*_{n}$ for $V^*$. 
Analogous coordinate charts $U^{\prime}_{i_1\cdots i_l}$ exist for 
$\mathbf{OG}^+(V \oplus V^*)$. For instance,   we have
\begin{equation}\label{ulxo}%\label{ulo}
U^{\prime}_{12\cdots l}:=\{\left(\begin{matrix}
  Z\\
-X^T\\
\end{matrix}\right.
  \hspace{-.11in}\begin{matrix}
  &\hfill\tikzmark{a}\\
  &\hfill\tikzmark{b}  
  \end{matrix} \,\,\,\,\,
  \begin{matrix}
  0\\
I_{(n-l)\times(n-l)}\\
\end{matrix}\hspace{-.11in}
\begin{matrix}
  &\hfill\tikzmark{c}\\
  &\hfill\tikzmark{d}
  \end{matrix}\hspace{-.11in}\begin{matrix}
  &\hfill\tikzmark{g}\\
  &\hfill\tikzmark{h}
  \end{matrix}\,\,\,\,
\begin{matrix}
I_{l\times l}\\
0\\
\end{matrix}\hspace{-.11in}
\begin{matrix}
  &\hfill\tikzmark{e}\\
  &\hfill\tikzmark{f}\end{matrix}\hspace{0.11in}\left.\begin{matrix}
  X\\
  W\\
\end{matrix}\right)\}
  \tikz[remember picture,overlay]   \draw[dashed,dash pattern={on 4pt off 2pt}] ([xshift=0.5\tabcolsep,yshift=7pt]a.north) -- ([xshift=0.5\tabcolsep,yshift=-2pt]b.south);\tikz[remember picture,overlay]   \draw[dashed,dash pattern={on 4pt off 2pt}] ([xshift=0.5\tabcolsep,yshift=7pt]c.north) -- ([xshift=0.5\tabcolsep,yshift=-2pt]d.south);\tikz[remember picture,overlay]   \draw[dashed,dash pattern={on 4pt off 2pt}] ([xshift=0.5\tabcolsep,yshift=7pt]e.north) -- ([xshift=0.5\tabcolsep,yshift=-2pt]f.south);\tikz[remember picture,overlay]   \draw[dashed,dash pattern={on 4pt off 2pt}] ([xshift=0.5\tabcolsep,yshift=7pt]g.north) -- ([xshift=0.5\tabcolsep,yshift=-2pt]h.south);
\end{equation}
around $L_{12\cdots l}$, where $Z$, $W$ are anti-symmetric such that $z_{ii}=w_{ii}=0$, $z_{ij}=-z_{ji}$, and $w_{ij}=-w_{ji}$.  We adopt the convention that $U^{\prime}_{0}$ is the coordinate chart for $l=0$.

Each point in $U^{\prime}_{12\cdots l}$ can be written as $L_A = \{ \ell + A(\ell): \ell \in L_{12\cdots l} \}\in U^{\prime}_{12\cdots l}$,
where $A$ can be represented by a skew-symmetric matrix  $\left(\begin{matrix}
    Z&X\\-X^T&W
\end{matrix}\right)$. 

Recall that the Pfaffian of a skew-symmetric $2r\times 2r$ matrix $X=(x_{ij})$ is defined by
\begin{equation*}
\operatorname{Pf}(X):=\sum_{\sigma}\operatorname{sgn}(\sigma)\, x_{\sigma(1)\sigma(2)}\cdots x_{\sigma(2r-1)\sigma(2r)}.
\end{equation*}
Here the sum is over all permutations $\sigma$ of $\{1,2,\cdots,2r\}$ such that $\sigma(2m-1)<\sigma(2m)$ for $1\le m\le r$ and $\sigma(2m-1)<\sigma(2m+1)$ for $1\le m\le r-1$; ${\rm sgn}(\sigma)=1$ for $\sigma$ even and ${\rm sgn}(\sigma)=-1$ for $\sigma$ odd.
\begin{lemma}\label{pure}
The pure spinor $\phi_{L_A}$ associated with $L_A$ up to a scalar is given by
\begin{equation*}%\label{puref}
\sum_{k=0}^{[n/2]}\sum_{\substack{1\leq i_1<\cdots <i_t\leq l\\l+1\leq j_1<\cdots <j_s\leq n \\ s+t=2k}} \operatorname{Pf}(A_{i_1\cdots i_tj_1\cdots j_s}) e_{i_1}\wedge e_{i_2}\wedge\cdots\wedge e_{i_t}\wedge e^*_{j_1}\wedge\cdots e^*_{j_s} \wedge e^*_{1}\wedge e^*_2\wedge\cdots\wedge e^*_{l}.
\end{equation*}
Here $A_{i_1\cdots i_tj_1\cdots j_s}$ is the submatrix of $A$ consisting of the $i_1^{th}, \cdots,i_t^{th},j_1^{th},\cdots, j_s^{th}$ rows and columns. We make the convention that the Pfaffian is $1$ for the empty matrix without rows or columns.
\end{lemma}
{\bf Proof of Lemma \ref{pure}.} Take an arbitrary $\alpha$ with $l+1\leq \alpha\leq n$. Write $\iota_{e_{\alpha}}\omega=\sum\nolimits_{i=1}^{l}x_{\alpha i} e_i + \sum\nolimits_{j=l+1}^{n} w_{\alpha j}e^*_j$. Computation yields that 
\begin{equation*}
\begin{split}
&\rho(e_{\alpha}+\iota_{e_{\alpha}}\omega)\phi_{L_A}=\rho(\iota_{e_{\alpha}}\omega)\sum_{\substack{1\leq i_1<\cdots <i_t\leq l\\l+1\leq j_1<\cdots <j_s\leq n \\ s+t=2[n/2]}} \operatorname{Pf}(A_{i_1\cdots i_tj_1\cdots j_s}) e_{i_1}\wedge\cdots\wedge e_{i_t}\wedge e^*_{j_1}\wedge\cdots\wedge e^*_{j_s} \wedge e^*_{1}\wedge\cdots\wedge e^*_{l}\\
&+\sum_{k=0}^{[n/2]-1}
\left\{\rho(e_{\alpha})\sum_{\substack{1\leq i_1<\cdots <i_t\leq l\\l+1\leq j_1<\cdots <j_s\leq n \\ s+t=2(k+1)}} \operatorname{Pf}(A_{i_1\cdots i_tj_1\cdots j_s}) e_{i_1}\wedge\cdots\wedge e_{i_t}\wedge e^*_{j_1}\wedge\cdots\wedge e^*_{j_s} \wedge e^*_{1}\wedge\cdots\wedge e^*_{l}\right.\\
&+\left.\rho(\iota_{e_{\alpha}}\omega)\sum_{\substack{1\leq i_1<\cdots <i_t\leq l\\l+1\leq j_1<\cdots <j_s\leq n \\ s+t=2k}} \operatorname{Pf}(A_{i_1\cdots i_tj_1\cdots j_s}) e_{i_1}\wedge\cdots\wedge e_{i_t}\wedge e^*_{j_1}\wedge\cdots\wedge e^*_{j_s} \wedge e^*_{1}\wedge\cdots\wedge e^*_{l}\right\}=:\Delta_1+\Delta_2.\\
\end{split}
\end{equation*}

Fix $1\leq i_1<\cdots <i_t\leq l$, $l+1\leq j_1<\cdots <j_{s+2}\leq n$ with $s+t=2k\leq 2[n/2]-2$. Without loss of generality,  assume $\alpha=j_{\lambda}$ with $l+1\leq\lambda\leq n$. The Laplace-type expansion yields that
\begin{equation*}
\begin{split}
&\operatorname{Pf}(A_{i_1\cdots i_tj_1\cdots j_{s+2}})=\sum\nolimits_{\beta=1}^{t}(-1)^{\beta+\lambda+t-1}x_{i_{\beta}j_{\lambda}} \operatorname{Pf}(A_{i_1\cdots\widehat{i_{\beta}}\cdots i_tj_1\cdots\widehat{j_{\lambda}}\cdots j_{s+2}})\\
&+\sum\nolimits_{\gamma=1}^{s+2}(-1)^{\gamma+\lambda-1}w_{j_{\gamma}j_{\lambda}}\operatorname{Pf}(A_{i_1\cdots i_tj_1\cdots \widehat{j_{\gamma}}\cdots\widehat{j_{\lambda}}\cdots j_{s+2}})+\sum\nolimits_{\gamma=1}^{s+2}(-1)^{\gamma+\lambda}w_{j_{\gamma}j_{\lambda}}\operatorname{Pf}(A_{i_1\cdots i_tj_1\cdots\widehat{j_{\lambda}}\cdots \widehat{j_{\gamma}}\cdots j_{s+2}}).     
\end{split}    
\end{equation*}
We conclude that $\Delta_2=0$. Similarly, we can show that $\Delta_1=0$.

The proof for $\rho(e^*_{\alpha}+\iota_{e^*_{\alpha}}\omega)\cdot\phi_{L_A}=0$ is the same. \,\,\,\,\,\endpf
\medskip

For $0\leq k\leq [n/2]$, we define index sets 
\begin{equation*}%\label{ikk}
\underline{\mathbb I}_{n}^{k}:=\big\{(i_1,\cdots,i_{2k})\in\mathbb Z^{2k}: 1\leq i_1<i_2<\cdots<i_{2k}\leq n\}.
\end{equation*}
We denote by $\underline{\mathcal S}_k$ the ideal sheaf of $\mathcal O_{\mathrm{OG}^+(n,2n)}$ generated by $\{\left(e_{\mathrm{OG}^+}\right)^*z_{I}:{I\in \underline{\mathbb I}_{n}^k}\}$, and let $\underline S_k$ be its vanishing locus.  Similarly, we can derive the following lemma.
\begin{lemma}\label{sepao}
For any integers $k,l$ such that $0\leq k\leq [n/2]$, $0\leq l\leq n$,  the restriction of $\underline{\mathcal S}_k$ to the affine coordinate chart $U^{\prime}_{12\cdots l}$ is generated by   
\begin{equation*}
   \left\{\begin{array}{ll}
    \{\operatorname{Pf}(Z_{i_1\cdots i_{l-2k}})\}_{1\leq i_1<\cdots <i_{l-2k}\leq l},\,\,&2k<l \\
    \{\operatorname{Pf}(W_{j_1\cdots j_{2k-l}})\}_{l+1\leq j_1<\cdots <j_{2k-l}\leq n}, \,\,&2k>l\\
       1,\,\,&2k=l\\
   \end{array}\right.. 
\end{equation*}
\end{lemma}

%For any permutation $\sigma$ of $\{0,1,\cdots,[\frac{n}{2}]\}$, let $g_0^{\sigma}:Y^{\sigma}_0\rightarrow \mathrm{OG}^+(n,2n)$ be the blow-up along $\underline{S}_{\sigma(0)}$, and inductively define $g^{\sigma}_{i+1}:Y^{\sigma}_{i+1}\rightarrow Y^{\sigma}_{i}$ as the blow-up along $(g^{\sigma}_0\circ\cdots\circ g^{\sigma}_i)^{-1}(\underline{S}_{\sigma(i+1)})$. Then $Y^{\sigma}_n\cong\mathcal {TO}_{n}$.

Similarly, the projection 
to the first factor of  $\mathbb {P}^{\underline N_{n}}\times\mathbb {P}^{\underline N^0_{n}}\times\cdots\times\mathbb {P}^{\underline N^{[n/2]}_{n}}$ induces a morphism $RO_{n}:\mathcal {TO}_{n}\rightarrow \mathrm{OG}^+(n,2n)$ that extends $\mathcal ({KO}_{n})^{-1}$ and factors as iterated blow-ups. Furthermore, there exists a unique equivariant ${\rm GL}_n\times\mathbb G_m$-action on $\mathcal {TO}_{n}$ compatible with (\ref{rspno}).

\section{Mille Cr\^epes Coordinates} \label{vander}
Following \cite{FW}, in this section, we will provide smooth atlases for $\mathcal {TL}_{n}$ and $\mathcal {TO}_{n}$ up to permutations of the basis. 

%We follow the method introduced in \cite{Fang,FW}, which is a generalization of the traditional one used in  \cite {Stu,Sev2,Van,Se2} by iteratively summing rank $1$ matrices. The core idea is to locally represent $\mathcal {TL}_{n}$ and $\mathcal {TO}_{n}$ as sequences of blow-ups of affine spaces along coordinate subspaces. The transition between successive blow-ups  is essentially the Gaussian elimination process. When the coordinate charts are constructed in reverse, all resulting matrices are summed up to consolidate the intermediate steps, which resembles stacking layers of paper-thin crepes and ganache on top of each other  (for this reason we call such coordinates {\it Mille Cr\^epes}).

\subsection{Coordinate charts for Lagrangian Grassmannians}\label{vanderl}
Fix an integer $l$ with $0\leq l\leq n$. Set $\mathbb {A}^{\frac{n(n+1)}{2}}:={\rm Spec}\,\mathbb Z[\overrightarrow A, X, \cdots,\overrightarrow B^{k},\cdots]$, where $X$ are the coordinates defined by (\ref{ulxl}), and
\begin{equation}\label{ulu}
\begin{split}
&\overrightarrow A:=\big((a_{n-l+i})_{1\leq i\leq l},(b_{i})_{l+1\leq i\leq n}\big),   \,\,\overrightarrow B^{k}:=\left\{\begin{array}{ll}
  (\xi_{j(l-k+1)})_{1\leq j\leq l-k} \,\,&{\rm for\,\,}1\leq k\leq l-1\\
   (\xi_{kj})_{k+1\leq j\leq n}\,\,&{\rm for\,\,}l+1\leq k\leq n-1 \\
   \end{array}\right..
\end{split}
\end{equation}

We define a morphism $\Gamma_l:\mathbb {A}^{\frac{n(n+1)}{2}}\rightarrow U_{12\cdots l}$ by
\begin{equation}\label{ws}
\left(
\begin{matrix}
\sum\limits_{k=1}^l\left(\prod\limits_{i=1}^{k}a_{n-l+i}\right)\cdot\Omega_k &0_{l\times(n-l)}&I_{l\times l}& X\\ -X^T&I_{(n-l)\times(n-l)}&0_{(n-l)\times l}&\sum\limits_{k=l+1}^{n}\left(\prod\limits_{i=l+1}^{k}b_{i}\right)\cdot\Omega_k\\
\end{matrix}\right)\,,    % \end{split}
\end{equation}
where $\Omega_k=\Xi_k^T\cdot\Xi_k$ and
\begin{equation}\label{xil}
\Xi_k:=\left\{\begin{array}{ll}
  (\xi_{1(l-k+1)},\xi_{2(l-k+1)},\cdots,\xi_{(l-k)(l-k+1)},\underset{\substack{\uparrow\\{(l-k+1)}^{\rm th}}}{1},\,\,\,0,\cdots,\underset{\substack{\uparrow\\{l}^{\rm th}}}{0}) \,\,&{\rm for\,\,}1\leq k\leq l\\(0,\cdots,0,\underset{\substack{\uparrow\\{(k-l)}^{\rm th}}}{1},\,\,\xi_{k(k+1)},\cdots,\underset{\substack{\uparrow\\{(n-l)}^{\rm th}}}{\xi_{kn}})\,\,&{\rm for\,\,}l+1\leq k\leq n \\
   \end{array}\right..   
\end{equation} 
Define a rational map   $J_l:\mathbb {A}^{\frac{n(n+1)}{2}}\dashrightarrow\mathbb {P}^{N_{n}}\times\mathbb {P}^{N^0_{n}}\times\cdots\times\mathbb {P}^{N^n_{n}}$ by $J_l:=\mathcal {KL}_{n}\circ \Gamma_l$.
Similarly to \cite[Lemma 3.2]{FW}, we can prove that $J_l$ is an embedding by (\ref{fskl}) and Lemma \ref{sepal}. Denoting by $\mathcal A_l$ the image of $\mathbb {A}^{\frac{n(n+1)}{2}}$, we call it
a {\it Type-I Mille Cr\^epes coordinate chart} of $\mathcal {TL}_n$.   
\begin{example}
Consider $LG(4,8)$ with $l=2$. Then	\begin{equation*}
\begin{split}
&\overrightarrow  A=(a_{3},a_{4},b_{3},b_{4}),\,\,X=\left(x_{13},x_{14},x_{23},x_{24}\right),\,\,\overrightarrow  B^1=(\xi_{12}),\,\,\,\,\overrightarrow  B^3=(\xi_{34}).\\
\end{split}
\end{equation*}
The morphism $\Gamma_2:\mathbb A^{10}\rightarrow U_{12}$ is given by 
\begin{equation*}
\left(
\begin{matrix}
a_{3}\cdot(\xi_{12}^2+a_{4})&a_{3}\cdot\xi_{12}&0&0&1&0&x_{13}&x_{14}\\
a_{3}\cdot\xi_{12}&a_{3}&0&0&0&1&x_{23}&x_{24}\\
-x_{13} &-x_{23} &1&0&0&0&b_{3}&b_{3}\cdot\xi_{34}\\
-x_{14}&-x_{24}&0&1&0&0&b_{3}\cdot\xi_{34}&b_{3}\cdot(\xi_{34}^2+b_{4})\\
	\end{matrix}
	\right).
	\end{equation*}
\end{example}

Before proceeding, we introduce another type of coordinate charts. For those who are only interested in the fields of characteristic not $2$, they may skip it. Define an index set 
\begin{equation*}
\mathbb J_l:=\left\{\left(i^+_1,\cdots,i^+_{\alpha},i^-_1,\cdots,i^-_{\beta}\right)\left|\,\,
\footnotesize\begin{matrix}
0\leq i^+_{1}-1<i^+_2-2<\cdots<i^+_{\alpha}-\alpha\leq l-1-\alpha\\
l\leq i^-_{1}-1<i^-_2-2<\cdots<i^-_{\beta}-\beta\leq n-1-\beta.
\end{matrix}\right.\right\}.
\end{equation*}
For each $\tau=(i^+_1,\cdots,i^+_{\alpha},i^-_1,\cdots,i^-_{\beta})\in\mathbb J_l$, define coordinates $X$ by (\ref{ulxo}), and
\begin{equation*}%\label{ulu}
\overrightarrow A:=\big((a_{n-l+i})_{1\leq i\leq l,\,\,i\neq i^+_1+1,i^+_2+1,\cdots,i^+_{\alpha}+1},(b_{i})_{l+1\leq k\leq n,\,\,i\neq i^-_1+1,i^-_2+1,\cdots,i^-_{\beta}+1}\big);
\end{equation*}
%\begin{equation}\label{ulu}\widetilde X:=\left(\begin{matrix}x_{1(s+l+1)}&\cdots &x_{1n}\\\vdots&\ddots&\vdots\\x_{l(s+l+1)}&\cdots &x_{ln}\\\end{matrix}\right)\,\,\,{\rm and}\,\,\,\widetilde Y:=\left(\begin{matrix}y_{(l+1)1}&\cdots& y_{(l+1)(s-p+l)}\\\vdots&\ddots&\vdots\\y_{p1}&\cdots& y_{p(s-p+l)}\\\end{matrix}  \right)  \,;\,\end{equation}
for $1\leq k\leq l-1$ and $k\neq i^+_1,i^+_1+1,\cdots,i^+_{\alpha},i^+_{\alpha}+1$, or $l+1\leq k\leq n-1$ and $k\neq i^-_1,i^-_1+1,\cdots,i^-_{\beta},i^-_{\beta}+1$, define $\overrightarrow B^{k}$ by (\ref{ulu});
for $k=i^+_{\gamma}$ with $1\leq\gamma\leq\alpha$,
\begin{equation}\label{k+}
\overrightarrow B^{k}:= \left(y_k,y_{k+1},(\xi_{j(l-k)})_{1\leq j\leq l-k-1},\,\,(\xi_{j(l-k+1)})_{1\leq j\leq l-k-1}\right),
\end{equation}
and for $k=i^-_{\gamma}$ with $1\leq\gamma\leq\beta$,
\begin{equation*}
\overrightarrow B^{k}:= \left(y_k,y_{k+1},(\xi_{kj})_{k+2\leq j\leq n},\,\,(\xi_{(k+1)j})_{k+2\leq j\leq n}\right).
\end{equation*}
Take the open subscheme
\begin{equation*}
\mathring{\mathbb {A}}:=\left\{\left.\mathfrak p\in {\rm Spec}\,\mathbb Z[\overrightarrow A, X, \cdots,\overrightarrow B^k,\cdots]\right|\,1-y_k\cdot y_{k+1}\notin\mathfrak p,\,\forall\,\,k\in\{i^+_1,\cdots,i^+_{\alpha},i^-_1,\cdots,i^-_{\beta}\}\right\}.
\end{equation*} 

Define a morphism $\Gamma_l^{\tau}:\mathring{\mathbb {A}}\rightarrow U_{12\cdots l}$ by
\begin{equation*}%\label{ws2}
\left(
\begin{matrix}
\sum\limits_{k=1}^l\left(\prod\limits_{\substack{1\leq i\leq k\\k\neq i^+_1+1,\cdots,i^+_{\alpha}+1}}a_{n-l+i}\right)\Omega_k &0_{l\times(n-l)}&I_{l\times l}& X\\ -X^T&I_{(n-l)\times(n-l)}&0_{(n-l)\times l}&\sum\limits_{k=l+1}^n\left(\prod\limits_{\substack{l+1\leq i\leq k\\k\neq i^-_1+1,\cdots,i^-_{\beta}+1}}b_i\right)\Omega_k\\
\end{matrix}\right)\,.   % \end{split}
\end{equation*}
Here for $k\in\{i^+_1+1,\cdots,i^+_{\alpha}+1,i^-_1+1,\cdots,i^-_{\beta}+1\}$, $\Omega_k$ is a zero matrix; for $1\leq k\leq l$ such that $k\neq i^+_1,i^+_1+1,\cdots,i^+_{\alpha},i^+_{\alpha}+1$, or $l+1\leq k\leq n$ such that $k\neq i^-_1,i^-_1+1,\cdots,i^-_{\beta},i^-_{\beta}+1$, $\Omega_k=\Xi_k^T\cdot\Xi_k$ where $\Xi_k$ is given by (\ref{xil}); for $k=i^+_{\gamma}$ with $1\leq\gamma\leq\alpha$,
\begin{equation}\label{2om+}
\Omega_k:=\left(\begin{matrix}
\Xi_k^T\cdot \left(\begin{matrix}
y_k&1\\1&y_{k+1}
\end{matrix}\right)^{-1}\cdot\Xi_k&\Xi_k^T&0\\
\Xi_k&\left(\begin{matrix}
y_k&1\\1&y_{k+1}
\end{matrix}\right)&0\\
0&0&0_{(k-1)\times(k-1)}\\
\end{matrix}\right)
\end{equation}
with 
$\Xi_k:=\left(\begin{matrix}
\xi_{1(l-k)}&\xi_{2(l-k)}&\cdots&\xi_{(l-k-1)(l-k)}\\
\xi_{1(l-k+1)}&\xi_{2(l-k+1)}&\cdots&\xi_{(l-k-1)(l-k+1)}\\
\end{matrix}\right)$;  
for $k=i^-_{\gamma}$ with $1\leq\gamma\leq\beta$,
\begin{equation*}%\label{2om-}
\Omega_k:=\left(\begin{matrix}
0_{(k-l-1)\times(k-l-1)}&0&0\\
0&\left(\begin{matrix}
y_k&1\\1&y_{k+1}
\end{matrix}\right)&\Xi_k\\
0&\Xi_k^T&\Xi_k^T\cdot \left(\begin{matrix}
y_k&1\\1&y_{k+1}
\end{matrix}\right)^{-1}\cdot\Xi_k
\end{matrix}\right)
\end{equation*}
with
$\Xi_k:=\left(\begin{matrix}
\xi_{k(k+2)}&\xi_{k(k+3)}&\cdots&\xi_{kn}\\
\xi_{(k+1)(k+2)}&\xi_{(k+1)(k+3)}&\cdots&\xi_{(k+1)n}
\end{matrix}\right)$.  
Denoting by $\mathcal A_l^{\tau}$ the image of $\mathring{\mathbb {A}}$ under the embedding $J_l^{\tau}:=\mathcal {KL}_{n}\circ \Gamma_l^{\tau}$, we call it a {\it Type-II Mille Cr\^epes coordinate chart} of $\mathcal {TL}_n$. 

\begin{example}
Consider $LG(4,8)$ with $l=0$, and $\tau=\left(1\right)\in\mathbb J_2$. Then	\begin{equation*}
\begin{split}
&\overrightarrow  A=(b_{1},b_{3},b_{4}),\,\,\overrightarrow  B^1=(y_1,y_2,\xi_{13},\xi_{14},\xi_{23},\xi_{24}),\,\,\,\,\,\,\overrightarrow  B^3=(\xi_{34}).\\
\end{split}
\end{equation*}
The map $\Gamma_0^{\tau}:\mathring{\mathbb A}\rightarrow U_{0}$ is given by 
\begin{equation*}
\footnotesize
\left(
\begin{matrix}
1&0&0&0&b_1y_1&b_1&\xi_{13}&\xi_{14}\\
0&1&0&0&b_1&b_1y_2&\xi_{23}&\xi_{24}\\
0&0&1&0&b_1\xi_{13}&b_1\xi_{23}&b_1\left(\frac{2\xi_{13}\xi_{23}-y_2\xi_{13}^2-y_1\xi_{23}^2}{1-y_1y_2}+b_3\right)&b_1\left(\frac{\xi_{13}\xi_{24}+\xi_{14}\xi_{23}-y_2\xi_{14}\xi_{13}-y_1\xi_{24}\xi_{23}}{1-y_1y_2}+b_3\xi_{34}\right)\\
0&0&0&1&b_1\xi_{14}&b_1\xi_{24}&b_1\left(\frac{\xi_{13}\xi_{24}+\xi_{14}\xi_{23}-y_2\xi_{14}\xi_{13}-y_1\xi_{24}\xi_{23}}{1-y_1y_2}+b_3\xi_{34}\right)&b_1\left(\frac{2\xi_{14}\xi_{24}-y_2\xi_{14}^2-y_1\xi_{24}^2}{1-y_1y_2}+b_3\left(\xi_{34}^2+b_4\right)\right)\\
	\end{matrix}
	\right).
	\end{equation*}
\end{example}

\begin{proposition}\label{tspnsmoothl}
$\mathcal {TL}_{n}$ is  smooth over ${\rm Spec}\,\mathbb Z$.   
\end{proposition}
{\noindent\bf Proof of Proposition \ref{tspnsmoothl}.}  
%We may assume that $2p\leq n\neq 2s$. 
It suffices to show that for any $0\leq l\leq n$,  $(RL_n)^{-1}(U_{12\cdots l})$ is covered by the union of $\mathcal A_l$ and $\mathcal A_l^{\tau}$ with $\tau\in\mathbb J_l$, up to permutations of $\{e_1,e_2,\cdots,e_n\}$.

Without loss of generality, we assume that $1\leq l\leq n-1$. Restrict $(\ref{sblow})$ to $U_{12\cdots l}$ with $\sigma$ defined by $\sigma(k)=l-k$ for $0\leq k\leq l$ and $\sigma(k)=k$ for $l+1\leq k\leq n$. Computation yields that $\mathcal S_{\sigma(0)}|_{U_{12\cdots l}}=\mathcal O_{U_{12\cdots l}}$ and $(g^{\sigma}_0)^*(\mathcal S_{\sigma(1)})|_{U_{12\cdots l}}$ is generated by $\{z_{ij}\}_{1\leq i\leq j\leq l}$. Hence $(g^{\sigma}_0\circ g^{\sigma}_{1})^{-1}(U_{12\cdots l})$ is isomorphic to ${\rm Spec\,}\mathbb Z[X,W]\times{\rm Bl}_{\{0\}}({\rm Spec\,}\mathbb Z[Z])$; we may view ${\rm Bl}_{\{0\}}({\rm Spec\,}\mathbb Z[Z])$ as a closed subscheme of ${\rm Spec\,}\mathbb Z[Z]\times {\rm Proj\,}\mathbb Z[V]$ defined by $z_{ij}v_{kl}=z_{kl}v_{ij}$, where $V:=(v_{ij})_{1\leq i\leq j\leq l}$. 

It is clear that the open subscheme of $(g^{\sigma}_0\circ g^{\sigma}_{1})^{-1}(U_{12\cdots l})$ defined by $u_{11}\neq 0$ can be parametrized by the following coordinate chart. Let $\mathbb {A}^{\frac{n(n+1)}{2}}:={\rm Spec}\,\mathbb Z[a_{n-l+1}, X, W, \overrightarrow B^{1},\overrightarrow B^*_1]$, where $X$, $W$ are defined by (\ref{ulxo}), $a_{n-l+1}$, $\overrightarrow B^1$ are defined by (\ref{ulu}), and $\overrightarrow B^{*}_1:=(u_{ij})_{1\leq i,j\leq l-1}$
with the convention $u_{ij}=u_{ji}$. Define a morphism $\Gamma_{l,1}:\mathbb A^{\frac{n(n+1)}{2}}\rightarrow U_{12\cdots l}$ by
\begin{equation}\label{wsc}
\left(
\begin{matrix}
a_{n-l+1}\left(\Omega_1+C_1^*\right) &0_{l\times(n-l)}&I_{l\times l}& X\\ -X^T&I_{(n-l)\times(n-l)}&0_{(n-l)\times l}&W\\
\end{matrix}\right)\,,    % \end{split}
\end{equation}
where $\Omega_1$ is defined by (\ref{xil}) and $C^*_1$ is defined by
\begin{equation}\label{w7}
C^*_1:=\left(\begin{matrix}
u_{11}&\cdots& u_{1(l-1)}&0\\
\vdots&\ddots&\vdots&\vdots\\
u_{(l-1)1}&\cdots& u_{(l-1)(l-1)}&0\\
0&\cdots&0&0\\
\end{matrix}\right).\\
\end{equation}
Setting $J_{l,1}:=\left(e,f^{l},f^{l-1}\right)\circ \Gamma_{l,1}$ (see (\ref{fskl}) for $f^k$), we obtain an embedding $J_{l,1}:\mathbb A^{\frac{n(n+1)}{2}}\rightarrow\mathbb {P}^{N_{n}}
\times\mathbb {P}^{N^l_{n}}\times\mathbb {P}^{N^{l-1}_{n}}$. We get the desired chart.  

Now consider the affine open subscheme of $(g^{\sigma}_0\circ g^{\sigma}_{1})^{-1}(U_{12\cdots l})$ defined by $v_{12}\neq 0$. It suffices to consider the open subscheme 
$\left\{\left.\mathfrak p\in (g^{\sigma}_0\circ g^{\sigma}_{1})^{-1}(U_{12\cdots l})\right|\,v_{12}^2-v_{11}v_{22}\notin\mathfrak p\right\}$
with the following parametrization. Define $\overrightarrow B^{*}_1:=(u_{ij})_{1\leq i,j\leq l-2}$
with $u_{ij}=u_{ji}$. Let $X$, $W$ be as in (\ref{ulxo}), and $a_{n-l+1}$, $\overrightarrow B^1$ as in (\ref{k+}). Take 
\begin{equation*}
\mathring{\mathbb A}_1:=\left\{\left.\mathfrak p\in
{\rm Spec}\,\mathbb Z[a_{n-l+1}, X, W, \overrightarrow B^{1},\overrightarrow B^*_1]\right|\,1-y_1y_2\notin\mathfrak p\right\}.
\end{equation*}
Define a morphism $\mathring\Gamma_{l}:\mathring{\mathbb A}_1\rightarrow U_{12\cdots l}$ by (\ref{wsc}), where $\Omega_1$ is given by (\ref{2om+}) and $C^*_1$ is defined as (\ref{w7}) by setting $u_{i(l-1)}=0$. Setting $\mathring J_{l}:=\left(e,f^{l},f^{l-1}\right)\circ\mathring \Gamma_{l}$, we get an embedding $\mathring J_{l}:\mathring{\mathbb A}_1\rightarrow\mathbb {P}^{N_{n}}
\times\mathbb {P}^{N^l_{n}}\times\mathbb {P}^{N^{l-1}_{n}}$. We get the desired chart.  

Proposition \ref{tspnsmoothl} can be proved by induction via Lemma \ref{sepal}. \,\,\,\,$\endpf$

\subsection{Coordinate charts for orthogonal Grassmannians}\label{vandero}
For any even integer $l$ with $0\leq l\leq n$, set $\mathbb {A}^{\frac{n(n-1)}{2}}:={\rm Spec}\,\mathbb Z[\overrightarrow A, X, \overrightarrow B^1,\cdots,\overrightarrow B^{\frac{l}{2}-1},\overrightarrow B^{\frac{l}{2}+1},\cdots,\overrightarrow B^{[\frac{n}{2}]}]$, where $X$ is defined by (\ref{ulxo}), $\overrightarrow A:=\big((a_{[\frac{n}{2}]-\frac{l}{2}+i})_{1\leq i\leq \frac{l}{2}},(b_{i})_{\frac{l}{2}+1\leq i\leq [\frac{n}{2}]}\big)$, and
\begin{equation*}
\overrightarrow B^{k}:=\left\{\begin{array}{ll}
  \left((\xi_{j(l-2k+1)})_{1\leq j\leq l-2k},\,\,(\xi_{j(l-2k+2)})_{1\leq j\leq l-2k}\right) \,\,&{\rm for\,\,}1\leq k\leq \frac{l}{2}-1\\
   \left((\xi_{(2k-1)j})_{2k+1\leq j\leq n},\,\,(\xi_{(2k)j})_{2k+1\leq j\leq n}\right)\,\,&{\rm for\,\,}\frac{l}{2}+1\leq k\leq [\frac{n}{2}] \\
   \end{array}\right..    
\end{equation*}

We define a morphism $\underline\Gamma_l:\mathbb {A}^{\frac{n(n-1)}{2}}\rightarrow U^{\prime}_{12\cdots l}$ by
\begin{equation*}%\label{ws2}
\left(
\begin{matrix}
\sum\limits_{k=1}^{\frac{l}{2}}\left(\prod\limits_{i=1}^ka_{[\frac{n}{2}]-\frac{l}{2}+i}\right)\underline\Omega_k &0_{l\times(n-l)}&I_{l\times l}& X\\ -X^T&I_{(n-l)\times(n-l)}&0_{(n-l)\times l}&\sum\limits_{k=\frac{l}{2}+1}^{[\frac{n}{2}]}\left(\prod\limits_{i=\frac{l}{2}+1}^kb_i\right)\underline\Omega_k\\
\end{matrix}\right)\,.  % \end{split}
\end{equation*}
Here for $1\leq k\leq \frac{l}{2}$,
\begin{equation*}%\label{2omo1}
\underline\Omega_k:=\left(\begin{matrix}
\underline\Xi_k^T\cdot \left(\begin{matrix}
0&1\\-1&0
\end{matrix}\right)\cdot\underline\Xi_k&-\underline\Xi_k^T&0\\
\underline\Xi_k&\left(\begin{matrix}
0&1\\-1&0
\end{matrix}\right)&0\\
0&0&0_{(2k-2)\times(2k-2)}\\
\end{matrix}\right)
\end{equation*}
with $\underline\Xi_k:=\left(\begin{matrix}
\xi_{(l-2k+1)1}&\xi_{(l-2k+1)2}&\cdots&\xi_{(l-2k+1)(l-2k)}\\
\xi_{(l-2k+2)1}&\xi_{(l-2k+2)2}&\cdots&\xi_{(l-2k+2)(l-2k)}\\
\end{matrix}\right)$; for $\frac{l}{2}+1\leq k\leq [\frac{n}{2}]$,
\begin{equation*}%\label{2omo2}
\underline\Omega_k:=\left(\begin{matrix}
0_{(2k-l-2)\times(2k-l-2)}&0&0\\
0&\left(\begin{matrix}
0&1\\-1&0
\end{matrix}\right)&\underline\Xi_k\\
0&-\underline\Xi_k^T&\underline\Xi_k^T\cdot \left(\begin{matrix}
0&1\\-1&0
\end{matrix}\right)\cdot\underline\Xi_k
\end{matrix}\right)
\end{equation*}
with $\underline\Xi_k:=\left(\begin{matrix}
\xi_{(2k-1)(2k+1)}&\cdots&\xi_{(2k-1)n}\\
\xi_{(2k)(2k+1)}&\cdots&\xi_{(2k)n}\\
\end{matrix}\right)$.
One can show that $\underline J_l:=\mathcal {KL}_{n}\circ \underline \Gamma_l$ is an embedding, whose image we denote by $\underline {\mathcal A}_l$.  We call $\left(\underline {\mathcal A}_l,(\underline J_l)^{-1}\right)$  a {\it Mille Cr\^epes coordinate chart} of $\mathcal {TO}_n$.   

\begin{example}
Consider $\mathrm{OG}^+(4,8)$ with $l=0$. The morphism $\underline\Gamma_0:\mathbb A^{10}\rightarrow U^{\prime}_{0}$ is given by 
\begin{equation*}
\left(
\begin{matrix}
1&0&0&0&0&b_1&b_1\xi_{13}&b_1\xi_{14}\\
0&1&0&0&-b_1&0&b_1\xi_{23}&b_1\xi_{24}\\
0&0&1&0&-b_1\xi_{13}&-b_1\xi_{23}&0&b_1\left(\xi_{13}\xi_{24}-\xi_{14}\xi_{23}+b_{3}\right)\\
0&0&0&1&-b_1\xi_{14}&-b_1\xi_{24}&b_1\left(\xi_{14}\xi_{23}-\xi_{13}\xi_{24}-b_{3}\right)&0\\
\end{matrix}\right).
\end{equation*}
\end{example}

Similarly to the proof of  Proposition \ref{tspnsmoothl}, we can conclude by Lemmas \ref{pure} and \ref{sepao} that
\begin{proposition}\label{tspnsmootho}
$\mathcal {TO}_{n}$ is  smooth over ${\rm Spec}\,\mathbb Z$.   
\end{proposition}

\section{Geometry via the Bia{\l}ynicki-Birula decomposition}\label{foliation}
In this section, combined with the Mille Cr\^epes coordinate charts, we 
investigate the geometry of $\mathcal {TL}_{n}$, $\mathcal {TO}_{n}$, $\mathcal {ML}_{n}$ and $\mathcal {MO}_{n}$ using the Bia{\l}ynicki-Birula decomposition.% (see \cite{Bi} for details). 

%Exploiting such idea in the Mille Cr\^epes coordinate charts, we  in this section. %, and then study dynamic behavior of the $\mathbb C^*$-actions $\psi_{s,p,n}$ and $\Psi_{s,p,n}$ on $G(p,n)$ and on $\mathcal T_{s,p,n}$. 

\subsection{Decomposition of  \texorpdfstring{$\mathcal {TL}_{n}$}{hh} and \texorpdfstring{$\mathcal {TO}_{n}$}{hh} }\label{foliationl}
The connected components of the fixed point scheme of $\mathrm{LG}(n,2n)$ under the $\mathbb G_m$-action (\ref{ltl}) are
\begin{equation*}%\label{vani}
\begin{split}
&\mathcal V_{(n-k,k)} :=\left\{\left. \left(
\begin{matrix}
0&X\\
Y&0\\
\end{matrix}\right)\in \mathrm{LG}(n,2n)\right\vert_{}\footnotesize\begin{matrix}
X\,\,{\rm is\,\,an\,\,}k\times n\,\,{\rm matrix\,\,of\,\,rank}\,\,k\,\\
Y\,\,{\rm is\,\,a\,\,}(n-k)\times n\,\,{\rm matrix\,\,of\,\,rank}\,\,(n-k)\\
\end{matrix}
\right\}\,,\,\,\,\,0\leq k\leq n,\\
\end{split}
\end{equation*}
and the stable, unstable subschemes are, respectively,
\begin{equation}\label{vani}
\begin{split}
&\mathcal V_{(n-k,k)}^+:= \left\{\left.\left(
\begin{matrix}
0&X\\
Y&W\\
\end{matrix}\right)\in \mathrm{LG}(n,2n)\right\vert_{}\footnotesize\begin{matrix}
X\,\,{\rm is\,\,an\,\,}k\times n\,\,{\rm matrix \,\,of\,\,rank\,\,}k\,\\
Y\,\,{\rm is\,\,a\,\,}(n-k)\times n\,\,{\rm matrix \,\,of\,\,rank\,\,}(n-k)\,\\
\end{matrix}\right\},\\
&\mathcal V_{(n-k,k)}^-:= \left\{\left.\left(
\begin{matrix}
Z&X\\
Y&0\\
\end{matrix}\right)\in \mathrm{LG}(n,2n)\right\vert_{}{\footnotesize\begin{matrix}
X\,\,{\rm is\,\,an\,\,}k\times n\,\,{\rm matrix \,\,of\,\,rank\,\,}k\,\\
Y\,\,{\rm is\,\,a\,\,}(n-k)\times n\,\,{\rm matrix \,\,of\,\,rank\,\,}(n-k)\,\\
\end{matrix}}\right\}\,,
\end{split}
\end{equation}
for $0\leq k\leq n$. One can conclude by Lemma \ref{sepal} that for $0\leq k\leq n$, the vanishing locus  $S_k$ is the union of $\overline{\mathcal V_{(n-k-1,k+1)}^+}$ and  $\overline{\mathcal V_{(n-k+1,k-1)}^-}$, where $\mathcal V_{(-1,n+1)}^+=\mathcal V_{(n+1,-1)}^-=\emptyset$ by convention.   
%\begin{definition}We call $\mathcal V_{(n,0)}$ the source and $\mathcal V_{(0,n)}$ the sink of $\mathrm{LG}(n,2n)$ under the action (\ref{ltl}).\end{definition}

Since $\mathbb G_m={\rm Spec}\,\mathbb Z[t,t^{-1}]$ in (\ref{ltl}) takes the form $b_{l+1}\mapsto t b_{l+1},\,\,a_{n-l+1}\mapsto t^{-1}a_{n-l+1}$ in any Mille Cr\^epes coordinate chart, the connected components of the fixed point scheme of $\mathcal {TL}_{n}$ are 
\begin{equation*}%\label{11l}
\mathcal D_{(n-k,k)}:=(RL_{n})^{-1}(\mathcal V_{(n-k,k)}), \,\,0\leq k\leq n,
\end{equation*}
with the stable, unstable subschemes given by
\begin{equation*}%\label{tschubertl}
\mathcal D^{\pm}_{(n-k,k)}:=(RL_{n})^{-1}(\mathcal V^{\pm}_{(n-k,k)}), \,\,0\leq k\leq n.
\end{equation*}
It is clear that $\mathcal D_{(n,0)}$ and $\mathcal D_{(0,n)}$ are smooth divisors of $\mathcal {TL}_{n}$, and  $\mathcal D_{(n-k,k)}$ are codimension two smooth closed subschemes for $1\leq k\leq n-1$. We define divisors 
\begin{equation*}%\label{boundaryl}
D_{k}^-:=\overline{\mathcal D_{(n-k+1,k-1)}^-},\,\,D_{k}^+:=\overline{\mathcal D_{(k-1,n-k+1)}^+},\,\,1\leq k\leq n.    
\end{equation*}
 
%See Figure (\ref{G2}) for illustration.

%Computation yields that
\begin{lemma}\label{snc}
The above divisors are smooth, prime, and have simple normal crossingss. %In particular,  $D_1^-=\mathcal D_{(p,0)}$,  $D_1^+=\mathcal D_{(p-r,r)}$.
%\begin{equation}\label{rr}   D_{1}^++D_{2}^++\cdots+D_{r}^++D_{1}^-+D_{2}^-+\cdots+D_{r}^-\,.\end{equation}
\end{lemma}

{\noindent\bf Proof of  Lemma \ref{snc}.} 
According to the proof of Proposition \ref{tspnsmoothl}, $\mathcal {TL}_{n}$ is covered by the Type-I and Type-II Mille Cr\^epes coordinate charts up to group actions.

Let $\mathcal A_l={\rm Spec}\,\mathbb Z[\overrightarrow A, X,\cdots,\overrightarrow B^{k},\cdots]$ be any Type-I coordinate chart. Then, $D_k^+\bigcap \mathcal A_l$ (resp. $D_k^-\bigcap \mathcal A_l$) is non-empty if and only if $n-l+1\leq k\leq n$ (resp. $l+1\leq k\leq n$), in which case it equals $\left\{a_{k}=0\right\}$ (resp. $\left\{b_{k}=0\right\}$).
Similar results hold for Type-II coordinate charts. \,\,\,$\endpf$
\medskip

Next, we construct the flat morphism $\mathcal {PL}_{n}:\mathcal {TL}_{n}\rightarrow D^-_1$. Let
$\mathfrak {P}:\mathcal {TL}_{n}\rightarrow \mathbb {P}^{N^0_{n}}\times\cdots\times\mathbb {P}^{N^n_{n}}$ be the restriction of  (\ref{nprojl}). Similarly to the proof of \cite[Lemmas 4.5 and 4.6]{FW}, it is clear that
\begin{lemma}\label{embl}
$\mathfrak P$ induces an embedding of $D^-_1$ into $\mathbb {P}^{N^0_{n}}\times\cdots\times\mathbb {P}^{N^n_{n}}$. The image of $\mathcal {TL}_{n}$ under $\mathfrak P$ is equal to the image of $D^-_1$ under $\mathfrak P$.
\end{lemma}

Define a morphism  $\mathcal {PL}_{n}:\mathcal {TL}_{n}\rightarrow D^-_1$ by
$\mathcal {PL}_{n}:=(\check {\mathfrak P})^{-1}\circ \mathfrak P$, where $\check {\mathfrak P}:D^-_1\rightarrow \mathfrak P\left(D^-_1\right)=\mathcal {ML}_{n}$ is the induced isomorphism.  One can show that $\mathcal {PL}_{n}$ is a retraction, and induces an isomorphism from $D_1^+$ to $D_1^-$.  Similarly to \cite[Lemma 4.8]{FW}, we can derive that $\mathcal {PL}_{n}$ is flat.

Recall the following realization of 
the spaces of complete quadrics.
\begin{example}\label{vcc}
Let $\mathbb P(S_{n\times n})$ be the projectivization of $S_{n\times n}$ the set of all $n\times n$ symmetric matrices. It is clear that $\mathbb P(S_{n\times n})$ is ${\rm GL}_n$-invariant. $\mathbb P(S_{n\times n})$ has $n$ ${\rm GL}_n$-orbits, whose closures $Z_{n}\supset Z_{n-1}\supset \cdots\supset Z_1$ are given by the condition
that $Z_i$ is the set of points corresponding to matrices of rank at most $i$.  Blowing up $\mathbb P(S_{n\times n})$ successively along the strict transform of
$Z_1,\cdots,Z_{n-1}$, we obtain the spaces of complete quadrics $Q(n-1)$. 
\end{example}

{\bf\noindent Proof of part (A) of Proposition \ref{red2l}.} Using the  Mille Cr\^epes coordinates,  we can show that $\widetilde {\mathbb P}(S_{n\times n})$ is isomorphic to $D^+_1$. By Lemma \ref{embl}, we have $D^+_1\cong\mathcal {ML}_{n}$. \,\,\,$\endpf$

\medskip

{\noindent\bf Proof of  Theorem \ref{gwondl}.} 
By Lemma \ref{cpcl}, $RL_{n}:\mathcal {TL}_{n}\rightarrow \mathrm{LG}(n,2n)$ extends $\mathcal {KL}^{-1}_{n}$. By Proposition \ref{tspnsmoothl}, $\mathcal {TL}_{n}$ is smooth. 
Property (A) follows from Lemma \ref{embl}. 

Recall that $\mathcal {PL}_{n}$ is an equivariant flat retraction whose restriction to $D^+_1$ is an isomorphism. 
Moreover, $\mathcal {PL}_{n}(\mathcal D_{(i-1,n+1-i)})=\mathcal {PL}_{n}(D^+_i)=\mathcal {PL}_{n}(D^-_{n+2-i})$ for $2\leq i\leq n$. Property (B) holds. 

It is easy to verify that, after base change to an algebraically closed field of characteristic not $2$, $\mathcal {TL}_n$ is covered by the Type-I Mille Cr\^epes coordinate charts up to the ${\rm GL}_n$-action. By Lemma \ref{snc}, we can conclude that the complement of the open ${\rm GL}_n$-orbit consists of $2n$ smooth prime divisors with simple normal crossings. Let $\mathcal A_l$ be a Type-I Mille Cr\^epes coordinate chart. Then $\mathfrak a, \mathfrak b\in \mathcal A_l$ are in the same orbit if and only if for each $1\leq i\leq n-l$ (resp. $n-l+1\leq i\leq n$), \begin{equation*}
b_{i}\in\mathfrak b\Longleftrightarrow b_{i}\in\mathfrak a\,\,({\rm resp.}\,\,a_{i}\in\mathfrak b\Longleftrightarrow a_{i}\in\mathfrak a).\\
\end{equation*}
Hence the
${\rm GL}_n$-orbit of $\mathcal {TL}_{n}$ one to one corresponds to (\ref{inrulel}). It is clear that the closure of each orbit is smooth. Moreover,  the complement of the open ${\rm GL}_n$-orbit of $D^-_1\cong\mathcal {ML}_{n}$ is the union of 
$\check  D_2,\,\check D_3,\, \cdots,\, \check D_n$. By Example \ref{vcc}, $D^-_1\cong\mathcal {ML}_{n}$ is wonderful. By Proposition \ref{gbl}, which will be proved in \S \ref{kcpt}, $\mathcal {TL}_n$ is toroidal. Property (C) holds.

%By Proposition \ref{mwond}, $D_1^-=\mathcal M_{s,p,n}$ is a wonderful $GL(s,\mathbb C)\times GL(n-s,\mathbb C)$-variety  with  $(r-1)$ $GL(s,\mathbb C)\times GL(n-s,\mathbb C)$-stable divisors $\check D_2,\cdots,\check D_r$.

We complete the proof of Theorem \ref{gwondl}. \,\,\,$\endpf$
\medskip

To prove part (A) of Theorem \ref{modulil}, we first prove the following lemma.
\begin{lemma}\label{moduli1l}
After base change to an algebraically closed field $\mathbb K$, the following holds. 
\begin{enumerate}[label=(\alph*)]

\item For each closed point $q\in \mathcal {ML}_{n}$, the fiber $Z_q:=(\mathcal {PL}_{n})^{-1}(q)$ consists of a chain of $\mathbb G_m$-stable smooth rational curves.

\item The restriction of $RL_{n}$ to $Z_q$ is an embedding.

\item Generically $RL_{n}(Z_q)$ is a smooth rational curve of degree $n$ with respect to the Pl\"ucker embedding.

\item For any closed points $q,q^{\prime}\in \mathcal {ML}_{n}$,  $RL_{n}(Z_q)=RL_{n}(Z_{q^{\prime}})$ if and only if $q=q^{\prime}$.

\end{enumerate} 
\end{lemma}
{\noindent\bf Proof of Lemma \ref{moduli1l}.} %Note that over $\mathbb K$, $\mathcal {TL}_n$ is covered by the Type-I Mille Cr\^epes coordinate charts up to the ${\rm GL}_n$-action.
Consider the Type-I Mille Cr\^epes coordinate chart $\mathcal A_{l}$ with $l=0$. It is clear that the fiber $Z_q$ of $\mathcal {PL}_{n}$ over a generic point $q$ is defined in $\mathcal A_l$ by fixing all the variables except $b_{1}$. Recall that $RL_n(Z_q\cap\mathcal A_l)$ is given by the projection to $\mathbb P^{N_{n}}$. Computing the corresponding determinants, we conclude that all determinants are polynomials in $b_{1}$, there is one  taking value $1$, and the highest degree of $b_{1}$ is $n$. The same argument works for any Type-II Mille Cr\^epes coordinate chart $\mathcal A_{0}^{\tau}$ with $\tau\in\mathbb J_0$. Property (c) follows. 

Take any Type-I Mille Cr\^epes coordinate chart $\mathcal A_l$ with $0\leq l\leq n$. Computation yields that if $l=0$ (resp. $n$), any fiber $Z_q$ restricted to $\mathcal A_l$ is defined by fixing all the variables except $b_{1}$ (resp. $a_1$); otherwise, any fiber $Z_q$ restricted to $\mathcal A_l$ is defined by fixing $a_{n-l+1}\cdot b_{l+1}$ and all the other variables. Then it is easy to verify that $Z_q$ is $1$-dimensional, reduced, and consists of a chain of $\mathbb G_m$-stable smooth rational curves; moreover, the restriction of $RL_n$ to $Z_q$ is an embedding. The same argument works for Type-II charts. We conclude Properties (a), (b).

Next, suppose that $RL_{n}(Z_q)=RL_{n}(Z_{q^{\prime}})$. We can write $Z_q$ as a chain of a rational curves $\cup_{i=1}^m\gamma_i$ with integers $0< k_1<\cdots<k_{m-1}<n$ such that 
$$\gamma_{1}\subset \overline{\mathcal D_{(n-k_1,k_1)}^-},\,\,\gamma_{i}\subset \overline{\mathcal D_{(n-k_i,k_i)}^-}\cap\overline{\mathcal D_{(n-k_{i-1},k_{i-1})}^+}\,\,{\rm for\,\,}2\leq i\leq m-1,\,\,{\rm and\,\,}\gamma_{m}\subset \overline{\mathcal D_{(n-k_{m-1},k_{m-1})}^+}.$$
Notice that the rational map $f^k$ given by (\ref{fskl}) is well-defined on $\gamma_i$ for any $1\leq i\leq m$ and $k_{i-1}\leq k\leq k_i$. We can thus conclude that $q=q^{\prime}$. Property (d) follows. 

We complete the proof of Lemma \ref{moduli1l}. \,\,\,$\endpf$
\medskip

{\noindent\bf  Proof of part (A) of Theorem \ref{modulil}.} By Property (B) in Theorem \ref{gwondl} and Properties (b), (d) in Lemma \ref{moduli1l}, there is a bijective morphism $\mathcal {ML}_{n}\longrightarrow \mathrm{LG}(n,2n)/ \! \! /\mathbb G_m$. % $\mathcal M_{s,p,n}$ is isomorphic to the normalization of $G(p,n)/ \! \! /\mathbb G_m$. 
We next show that $\mathrm{LG}(n,2n)/ \! \! /\mathbb G_m$ is smooth following the idea in the proof of \cite[Proposition 6.3]{ORCW}.

Denote by $C_g$ the closure of the general orbit of the action in $\mathrm{LG}(n,2n)$, and by $[C_g]$ the corresponding point in the Hilbert scheme ${\rm Hilb}(\mathrm{LG}(n,2n))$. Notice that the tangent bundle $T_{\mathrm{LG}(n,2n)}$ is globally generated. Then  $H^1(\mathbb P^1,\mu^*T_{\mathrm{LG}(n,2n)})=0$ for any morphism $\mu:\mathbb P^1\to \mathrm{LG}(n,2n)$, and 
$H^1(C_g,\mathcal N_{C_g/\mathrm{LG}(n,2n)})=H^1(C_g,{\left.T_{\mathrm{LG}(n,2n)}\right|}_{C_g})=0$ by Property (c) in Lemma \ref{moduli1l}.  Hence ${\rm Hilb}(\mathrm{LG}(n,2n))$ is smooth at $[C_{g}]$, and there is a unique irreducible component of ${\rm Hilb}(\mathrm{LG}(n,2n))$ containing it, which we denote by $\mathcal H$. Consider the induced $\mathbb G_m$-action on $\mathcal H$. By \cite[Theorem 4.1]{Bi}, there is a unique connected component of the fixed point scheme $\mathcal M \subset \mathcal H$ containing $[C_g]$ such that $\mathcal M$ is smooth at $[C_g]$. It is clear that the generic point of $\mathcal M$ parametrizing a smooth $\mathbb G_m$-invariant rational curve passing through  $\mathcal V_{n,0}$ and $\mathcal V_{0,n}$. Then $\mathcal M\cong \mathrm{LG}(n,2n)/ \! \! /\mathbb G_m$. Now by \cite[Theorem 2.5]{Bi} again,
it suffices to show that $\mathcal H$ is smooth at every point of $\mathcal M$.

Since the morphism $\mathcal {ML}_{n}\longrightarrow \mathrm{LG}(n,2n)/ \! \! /\mathbb G_m$ is surjective, for every $[C]\in \mathcal M$, the corresponding subscheme $C$ is a chain of smooth rational curves with simple normal crossingss by Property (a) in Lemma \ref{moduli1l}. Hence the conormal sheaf $\mathcal I_{C,\mathrm{LG}(n,2n)}/\mathcal I^2_{C,\mathrm{LG}(n,2n)}$ is locally free. Applying the functor ${\rm Hom}(\,\cdot\,,\mathcal O_C)$ to the short exact sequence
$$
0\rightarrow{\mathcal I_{C,\mathrm{LG}(n,2n)}/\mathcal I^2_{C,\mathrm{LG}(n,2n)}}\rightarrow{\left.\Omega_{\mathrm{LG}(n,2n)}\right|}_{C}\rightarrow{\Omega_C}\rightarrow0,$$
we get that ${\rm Ext}^1(\mathcal I_{C,\mathrm{LG}(n,2n)}/\mathcal I^2_{C,\mathrm{LG}(n,2n)},\mathcal O_C)$ is a quotient of ${\rm Ext}^1({\left.\Omega_{\mathrm{LG}(n,2n)}\right|}_{C},\mathcal O_C)=H^1(C,{\left.T_{\mathrm{LG}(n,2n)}\right|}_{C})$. Moreover, we can prove by induction that $H^1(C,{\left.T_{\mathrm{LG}(n,2n)}\right|}_{C})=0$
as in the proof of \cite[Lemma~10]{FP}. Hence $\mathcal M$ is smooth at $[C]$ by \cite[Theorem~I.2.8]{Ko}.

We complete the proof of Theorem \ref{modulil}.\,\,\,$\endpf$

\medskip

Similarly, for $1\leq k\leq [\frac{n}{2}]$, we can define  smooth prime divisors with simple normal crossingss:
\begin{equation*}%\label{boundaryo}
D_{k}^-:=\overline{(RO_{n})^{-1}(\underline{\mathcal V}_{(n-2k+2,2k-2)}^-)},\,\,D_{k}^+:=\overline{(RO_{n})^{-1}(\underline{\mathcal V}_{(n-2[\frac{n}{2}]+2k-2,2[\frac{n}{2}]-2k+2)}^+)}, 
\end{equation*} 
where $\underline{\mathcal V}_{(n-2k,2k)}^{\pm}$ are stable/unstable subschemes
of $\mathrm{OG}^+(n,2n)$, defined similarly to those in equation (\ref{vani}).  Note that by Lemmas \ref{pure},  \ref{sepao},  $\underline S_k$ is the union of $\overline{\underline{\mathcal V}_{(n-2k-2,2k+2)}^+}$ and  $\overline{\underline{\mathcal V}_{(n-2k+2,2k-2)}^-}$.

We can define the morphism  $\mathcal {PO}_{n}:\mathcal {TO}_{n}\rightarrow D^-_1$ in the same way.

\medskip

%Similar to the proof of Lemma \ref{moduli1l}, we can show that \begin{lemma}\label{moduli1o}Base change to an algebraically closed field $\mathbb K$. The following holds.\begin{enumerate}[label=(\alph*)]\item For each closed point $q\in \mathcal {MO}_{n}$, the fiber $Z_q:=(\mathcal {PO}_{n})^{-1}(q)$ consists of a chain of $\mathbb G_m$-stable smooth rational curves.\item The restriction of $RO_{n}$ to $Z_q$ is an embedding.\item Generically $RO_{n}(Z_q)$ is a smooth rational curve of degree $[\frac{n}{2}]$ with respect to the spinor embedding.\item For any closed points $q,q^{\prime}\in \mathcal {MO}_{n}$,  $RO_{n}(Z_q)=RO_{n}(Z_{q^{\prime}})$ if and only if $q=q^{\prime}$.\end{enumerate} \end{lemma}

{\bf\noindent Proof of Theorem \ref{gwondo}, and part (B) of Proposition \ref{red2l}.} The proof is the same as that for Theorem \ref{gwondl} and part (A) of Proposition \ref{red2l}. \,\,\,$\endpf$ 
\medskip

{\bf\noindent Proof of Theorem \ref{modulil}.} Part (B) can be proved similarly. The smoothness of the universal families associated with ordinary Grassmannians follows from \cite[Theorems 1.3, 1.8]{FW}. \,\,\,$\endpf$

\medskip

\subsection{Resolution of the Landsberg--Manivel birational maps}\label{landmani}

The Landsberg--Manivel birational map $\mathcal {LML}:\mathbb P^{\frac{n(n+1)}{2}}\dashrightarrow
\mathrm{LG}(n,2n)\subset \mathbb P^{N_{n}}$ takes the form 
$e_{\mathrm{LG}}\circ LML$, where $LML:\mathbb P^{\frac{n(n+1)}{2}}\dashrightarrow\mathrm{LG}(n,2n)$ is given by
\begin{equation}\label{lmm}    
[x_{00}, x_{ij} (1\leq i\leq j\leq n)]\xdashmapsto{\,\,\,}\big(x_0\cdot I_{n\times n}\hspace{-.12in} \begin{matrix} &\hfill\tikzmark{c}\\ 
&\hfill\tikzmark{d} \end{matrix} \hspace{.12in} (x_{ij})_{1\leq i,j\leq n}\big)
\tikz[remember picture,overlay]   \draw[dashed,dash pattern={on 4pt off 2pt}] ([xshift=0.5\tabcolsep,yshift=3pt]c.north) -- ([xshift=0.5\tabcolsep,yshift=1.5pt]d.south);
\end{equation}
with the convention $x_{ij}=x_{ji}$.

Similarly to \cite{Ka}, we define $X^{(0)}:={\rm Proj}\,\mathbb Z[x_{00}, x_{ij} (1\leq i\leq j\leq n)]$, and its closed subschemes
\vspace{-.08in}
\begin{equation}\label{lblow}
\begin{tikzcd}
&Y^{(0)}_0\ar[hook]{r}&Y^{(0)}_1\ar[hook]{r}&\cdots\ar{r}&Y^{(0)}_{n-1}&\\
&&Z^{(0)}_{n-1}\ar[hook]{u}\ar[hook]{r}&\cdots\ar[hook]{r}&Z^{(0)}_1\ar[hook]{u}\ar[hook]{r}&Z^{(0)}_0\\
\end{tikzcd}\vspace{-20pt}
\end{equation}
by $Y^{(0)}_l:=V(\mathcal I^{(0)}_l)$, $Z^{(0)}_l:=V(\mathcal J^{(0)}_l)$, $0\leq l\leq n-1$. Here $\mathcal I^{(0)}_l$, $0\leq l\leq n-1$, is the homogeneous ideal generated by all $(l+1)\times(l+1)$ subdeterminants of the matrix $(x_{ij})_{1\leq i,j\leq n}$ with the convention $x_{ji}=x_{ij}$;  $\mathcal J^{(0)}_0:=(x_{00})$, $\mathcal J^{(0)}_l:=(x_{00})+\mathcal I^{(0)}_{n-l}$ for $1\leq l\leq n-1$. For $1\leq k \leq n-1$, inductively define %let the scheme $X^{(k)}$ together with closed subschemes $Y^{(k)}_l$, $Z^{(k)}_l$ $(0\leq l\leq p-1)$ defined as follows.
$X^{(k)}\rightarrow X^{(k-1)}$ to be the blow-up of $X^{(k-1)}$ along the closed subscheme $Y^{(k-1)}_{k-1}\cup Z^{(k-1)}_{n-k}$. Define the subscheme $Y^{(k)}_{k-1}\subset X^{(k)}$ (resp.~$Z^{(k)}_{n-k}\subset X^{(k)}$) to be the inverse image of $Y^{(k-1)}_{k-1}$ (resp.~$Z^{(k-1)}_{n-k}$) under the morphism $X^{(k)}\rightarrow X^{(k-1)}$, and for $l\neq k-1$ (resp.~$l\neq n-k$) define the subscheme $Y^{(k)}_l\subset X^{(k)}$ (resp.~$Z^{(k)}_l\subset X^{(k)}$) to be the strict transform of $Y^{(k-1)}_l\subset X^{(k-1)}$ (resp. $Z^{(k-1)}_l\subset X^{(k-1)}$). Set ${\rm KAL}_{n}:=X^{(n-1)}$, $Y_l:=Y^{(n-1)}_l$, $Z_l:=Z^{(n-1)}_l$, and denote the corresponding blow-up by
\begin{equation}\label{kmorpl}
\mathcal {KAL}:{\rm KAL}_{n}\longrightarrow X^{(0)}\cong\mathbb P^{\frac{n(n+1)}{2}}.
\end{equation}

\begin{lemma}\label{g2kausz}
${\rm KAL}_{n}\cong\mathcal {TL}_{n}$ over ${\rm Spec}\,\mathbb Z$. The blow-up center of each
$X^{(k)}\rightarrow X^{(k-1)}$, $1\leq k \leq n-1$, consists of two disjoint irreducible smooth subschemes. 
\end{lemma}
{\noindent\bf Proof of Lemma \ref{g2kausz}.} Denote by $V_{ij}$ the open subscheme of $\mathbb P^{\frac{n(n+1)}{2}}$
given by $x_{ij}\neq 0$. Decomposing $RL_n$ into iterated blow-ups as in \S \ref{iterated}, we have $(\mathcal {KAL})^{-1}(V_{00}) \cong (RL_n)^{-1}(U_0)$.

Consider $V_{11}$ with coordinates $y_{00}:=x_{00}/x_{11}$, $y_{ij}:=x_{ij}/x_{11}$. Computation yields that 
\begin{equation}\label{cen0}
\begin{split}
&Y_0^{(0)}\cap V_{11} = \emptyset,\,\,Z_{n-1}^{(0)}\cap V_{11}= V\left(y_{00},\left|\begin{matrix}
1&y_{1j}\\    y_{1i}&y_{ij}
\end{matrix}\right|,\,\,2\leq i\leq j\leq n\right).
\end{split}    
\end{equation}
Make a change of coordinates by $u_{ij}:=y_{ij}-y_{1j}y_{1i}$, $2\leq i\leq j\leq n$. The blow-up
$X^{(1)} \rightarrow X^{(0)}$ over $V_{11}$  is isomorphic to the blow-up of ${\rm Spec}\,\mathbb Z[y_{00},y_{12},\cdots,y_{1n}, u_{22},u_{23},\cdots, u_{nn}]$ in the ideal $(y_{00},u_{22},u_{23},\cdots, u_{nn})$, which we denote by ${\rm Bl}_{(Y_0^{(0)}\cup Z_{n-1}^{(0)})\cap V_{11}}V_{11}$,

Take a chart $V_{11,00}^{(1)}:={\rm Spec}\,\mathbb Z[y_{00},y_{12},\cdots$, $y_{1n},v_{22},\cdots]$ of ${\rm Bl}_{(Y_0^{(0)}\cup Z_{n-1}^{(0)})\cap V_{11}}V_{11}$ defined by $v_{ij}:=u_{ij}/y_{00}$, $2\leq i\leq j\leq n$. Then, the strict transforms of 
$Y^{(0)}_1$, $Z^{(0)}_{n-2}$ in $V_{11,00}^{(1)}$ are
\begin{equation}\label{cen1}
V\left( 
v_{ij},\,\,2\leq i\leq j\leq n\right),\,\,V\left(\left|\begin{matrix}
v_{ik}&v_{il}\\
v_{jk}&v_{jl}
\end{matrix} \right| ,\,\,2\leq i,j,k,l\leq n\right),   
\end{equation}
respectively. It is easy to verify that the pull-back of $V_{11,00}^{(1)}$ to $\operatorname{KA}_n$ is isomorphic to $(RL_n)^{-1}(U_{1})$.

Take a chart $V_{11,22}^{(1)}:={\rm Spec}\,\mathbb Z[\underline y_{00},y_{12},\cdots,y_{1n}, u_{22}, \underline y_{23},\cdots]$ of ${\rm Bl}_{(Y_0^{(0)}\cup Z_{n-1}^{(0)})\cap V_{11}}V_{11}$ defined by 
$\underline y_{00}:=y_{00}/u_{22}$, $\underline y_{2j}:=u_{2j}/u_{22}$ for $3\leq j\leq n$, and $\underline y_{ij}:=u_{ij}/u_{22}-u_{2i}u_{2j}/u_{22}^2$ for $3\leq i\leq j\leq n$. Computation yields that the strict transforms of 
$Y^{(0)}_1$, $Z^{(0)}_{n-2}$ in $V_{11,22}^{(1)}$ are 
\begin{equation}\label{cen2}
Y^{(1)}_1\cap V_{11,22}^{(1)}=\emptyset,\,\,Z^{(1)}_{n-2}\cap V_{11,22}^{(1)}=V\left(  \underline y_{00},\,\, \left|\begin{matrix}
1&\underline y_{2j}\\ \underline y_{2i}&\underline y_{ij}
\end{matrix}\right|,\,\,3\leq i\leq j\leq n\right),   
\end{equation}
respectively, which is similar to (\ref{cen0}).

For $V_{12}$, and the chart $V_{11,23}^{(1)}:={\rm Spec}\,\mathbb Z[\underline y_{00},y_{12},\cdots,y_{1n},  \underline y_{22},u_{23},$ $ \underline y_{24},\cdots]$ of ${\rm Bl}_{(Y_0^{(0)}\cup Z_{n-1}^{(0)})\cap V_{11}}V_{11}$ defined by $\underline y_{00}=y_{00}/u_{23}$, $\underline y_{ij}=u_{ij}/u_{23}$ for $2\leq i\leq j\leq n$ such that $(i,j)\neq(2,3)$,  we can use the Type-II  Mille Cr\^epes coordinate charts to compute the blow-ups.

Continuing the above process, we can prove that ${\rm KAL}_{n}\cong\mathcal {TL}_{n}$ by induction. Moreover, by (\ref{cen0}), (\ref{cen1}), (\ref{cen2}), we can show that the blow-up center of each
$X^{(k)}\rightarrow X^{(k-1)}$, $1\leq k \leq n$, consists of two disjoint irreducible smooth subschemes. 

The proof of
Lemma \ref{g2kausz} is complete. \,\,\,$\endpf$
\medskip

{\noindent\bf Proof of part (A) of Theorem \ref{g4kausz}.} 
It is easy to verify that the blow-up $\mathcal {KAL}$ is identical to the projection $\mathcal {TL}_n\rightarrow \mathbb P\left((\bigwedge^{n}V\otimes\bigwedge^{0}V^*)\oplus(\bigwedge^{n-1}V\otimes\bigwedge^{1}V^*)\right)$. If we only blow up $Z_{n-k}^{k-1}$, $1\leq k\leq n-1$, in the construction (\ref{lblow}), (\ref{kmorpl}), we will get a smooth scheme $tl_n$. One can show that $tl_n$ is the blow-up of $\mathrm{LG}(n,2n)$ iteratively along the total transforms of $\overline{\mathcal V_{(0,n)}^+}$, $\overline{\mathcal V_{(1,n-1)}^+}$, $\cdots$, $\overline{\mathcal V_{(n-2,2)}^+}$. We conclude part (A) of Theorem \ref{g4kausz} by Lemma \ref{g2kausz}. \,\,\,$\endpf$
\medskip

The Landsberg--Manivel birational map $\mathcal {LMO}:\mathbb P^{\frac{n(n-1)}{2}}\xdashrightarrow{\,\,\,}
\mathrm{OG}^+(n,2n)\subset \mathbb P^{\underline N_{n}}$ takes the form 
$e_{\mathrm{OG}^+}\circ LMO$, where $LMO$ is defined similarly to (\ref{lmm}) with convention $x_{ij}=-x_{ji}$. The blow-up ${\mathcal{ KAO}_{n}}$ can be constructed analogously by replacing subdeterminants with sub-Pfaffians in (\ref{lblow}). Part (B) of Theorem \ref{g4kausz} can be proved similarly.

\section{Differential Geometric Properties}\label{kcpt}

For a smooth projective variety $X$, write $N^1(X)$ for the group of divisor classes modulo numerical equivalence and set  $N^1(X)_{\mathbb R}=N^1(X)\otimes_{\mathbb Z}\mathbb R$. 
%The effective cone of $X$ is the convex cone $\operatorname{Eff}(X)\subset N^{1}(X)_{\mathbb R}$ generated by classes of effective divisors. The nef cone of $X$ is the closed convex cone $\mathrm{Nef}(X)\subset N^{1}(X)_{\mathbb R}$ generated by classes of numerically effective divisors. The movable cone of $X$ is the convex cone $\mathrm{Mov}(X)\subset N^{1}(X)_{\mathbb R}$ generated by classes of movable divisors. %These are Cartier divisors whose stable base locus has codimension at least two in $X$. 
Denote by $\mathrm{Nef}(X), \mathrm{Mov}(X),\mathrm{Eff}(X)\subset N^1(X)_{\mathbb{R}}$ be its nef, movable, and effective cone, respectively.
%We have inclusions $\operatorname{Nef}(X) \subset \overline{\operatorname{Mov}(X)}\subset\overline{\operatorname{Eff}(X)}$. %We refer to \cite[Chapter 1]{Deb} for a comprehensive treatment of these topics.

%We define \begin{equation*}\begin{split}\mathcal D(X) & = \{D \subset X \text{ $B$-stable prime divisor}\}, \\\Delta(X) & = \{D \in \mathcal D(X) \mid \text{$D$ is not $G$-stable}\}, \\\mathring{\Delta}(X)&=\{\text{$B$-stable prime divisors containing no $G$-orbit }\}.    \end{split}    \end{equation*}

%\begin{theorem}[Theorem 3.1.3 in \cite{Per}] The following holds.\begin{enumerate} \item Any Cartier divisor on $X$ is linearly equivalent to $\sum_{D \in \mathring{\Delta}(X)} n_D D$ with $n_D \in \mathbb{Z}$. \item A Cartier divisor is globally generated (resp. ample) if and only if it is a linear combination $\sum_{D \in \mathring{\Delta}(X)} n_D D$ with $n_D \geq 0$ (resp. $n_D > 0$) for all $D$.\end{enumerate}\end{theorem}

%$X$ is called a Mori dream space if the following conditions hold. \begin{enumerate}[label={(\arabic*)}]\item $\operatorname{Pic}(X)$ is finitely generated.\item $\operatorname{Nef}(X)$ is generated by the classes of finitely many semi-ample divisors.\item There is a finite collection of small $\mathbb{Q}$-factorial modifications $f_i : X \rightarrow X_i$, such that each $X_i$ satisfies (2), and $\operatorname{Mov}(X) = \bigcup_i f_i^*(\operatorname{Nef}(X_i))$.\end{enumerate}

In this section, we will work over an algebraically closed field $\mathbb K$ of characteristic not $2$, until specifically mentioned.

\subsection{Effective and nef cones of \texorpdfstring{$\mathcal {TL}_n$}{dd}}\label{kcptl}
For convenience, we denote by $K$ the canonical bundle of $\mathcal {TL}_{n}$, and by $H$ the line bundle $(RL_{n})^*\big(\mathcal O_{\mathrm{LG}(n,2n)}(1)\big)$, where $\mathcal O_{\mathrm{LG}(n,2n)}(1)$ is the hyperplane bundle. 
\begin{lemma}\label{excep}
The exceptional divisor of the blow-up ${RL}_{n}$ is $D_{1}^++\cdots+D_{n-1}^++D_{1}^-+\cdots+D_{n-1}^-$.
The Picard group $\operatorname{Pic}(\mathcal {TL}_{n})$ is freely generated by $H,D^+_1,\cdots,D^+_{n-1},D^-_1,\cdots,D^-_{n-1}$. Moreover, 
\begin{equation*}
\begin{split}
K&=-(n+1) H+\sum\limits_{i=1}^{n-1}\left(\frac{(n+3-i)(n-i)}{2}\right) (D_{i}^- +D_{i}^+)=-\sum\limits_{m=1}^{n-1}B_m-\sum\limits_{i=1}^{n}(D_{i}^- +D_{i}^+)\,.   
\end{split}
\end{equation*}
\end{lemma}

{\noindent\bf Proof of Lemma \ref{excep}.} As in \S \ref{iterated},  we construct $RL_n$ as iterated blow-ups: first blow up  $\overline{\mathcal V_{(n,0)}^-}$, then iteratively the total transforms of $\overline{\mathcal V_{(n-1,1)}^-}$, $\overline{\mathcal V_{(n-2,2)}^-}$, $\cdots$, $\overline{\mathcal V_{(1,n-1)}^-}$; next, iteratively blow up the total transforms of $\overline{\mathcal V_{(0,n)}^+}$, $\overline{\mathcal V_{(1,n-1)}^+}$, $\cdots$, $\overline{\mathcal V_{(n-1,1)}^+}$. Notice that the center of each blow-up is a union of a smooth subscheme and a divisor. We conclude Lemma \ref{excep}.\,\,\,
$\endpf$
\medskip

%Recall that any $(A,\lambda)\in{\rm GL}_n\times\mathbb G_m$ acts on $\mathrm{LG}(n,2n)$ via right matrix multiplication by \begin{equation}\label{glg}\begin{pmatrix}A^{-T}&0\\0&\lambda\cdot A\\\end{pmatrix}.   \end{equation} 
Denote by $B$ the Borel subgroup of ${\rm GL}_n$ consisting of upper triangular matrices. Let $T$  be  a maximal torus of $B$. It is easy to verify that over $\mathbb K$, $\mathcal {TL}_{n}$ is a spherical ${\rm GL}_n$-variety.

Define irreducible divisors of $\mathrm{LG}(n,2n)$ by $b_{k}:=\left\{x\in \mathrm{LG}(n,2n) |\,P_{I_k}(x)=0\,\right\}$, $0\leq k\leq n$, where $I_k:=(k+1,k+2,\cdots,n+k)\in\mathbb I^k_{n}$. Let $B_{k}\subset\mathcal {TL}_{n}$,
$0\leq k\leq n$, be the strict transform of $b_k$ under ${RL}_{n}$. In the Type-I Mille Cr\^epes coordinate chart $\mathcal A_l$, we have  \begin{equation}\label{te}
(RL_{n})^*(P_{I_k})=\left\{ \begin{array}{cl}
    1\,\,\,\,\,\,\,\,\,\,\,\,\,\,&{\rm if}\,\,k=l\\
    \prod\nolimits_{i=1}^{l-k} a^{l-k+1-i}_{n-l+i}&{\rm if}\,\,k\leq l-1\\
    \prod\nolimits_{i=l+1}^{k}b^{k+1-i}_{i}\,\,\,\,\,&{\rm if}\,\,k\geq l+1\,\,\\
    \end{array}\right..
\end{equation}
Hence, \begin{equation}\label{bstl}
\begin{split}
&B_k= H-\sum\nolimits_{i=1}^{n-k}(n-k+1-i) D^+_i-\sum\nolimits_{i=1}^k(k+1-i) D^-_i\,\,\,\,{\rm for}\,\,1\leq k\leq n,\\
&B_0= D^+_n=H-\sum\nolimits_{i=1}^{n-1}(n+1-i) D^+_i,\,\,B_n= D^-_n=H-\sum\nolimits_{i=1}^{n-1}(n+1-i) D^-_i\,.
\end{split}
\end{equation}

Recall that in the theory of spherical varieties, $B$-invariant but not $G$-invariant  divisors are called colors;  a spherical variety is said to be toroidal if every irreducible $B$-stable divisor containing a $G$-orbit is itself $G$-stable (see, for instance, \cite{Per}). 
\begin{proposition}\label{gbl}
The irreducible $G$-invariant divisors of $\mathcal {TL}_{n}$ are $D_{1}^-,  \cdots, D_{n}^-, D_{1}^+, \cdots, D_{n}^+$.
The colors of $\mathcal {TL}_{n}$ are $B_1, \cdots, B_{n-1}$. Moreover, $\mathcal {TL}_{n}$ is toroidal.
\end{proposition}
{\noindent\bf Proof of Proposition \ref{gbl}.} %The first statement holds by Theorem \ref{gwondl}.
If $\mathfrak D$ is a $B$-invariant divisor not among those listed above,  $RL_n(\mathfrak D)$ is a $B$-invariant divisor  containing the point $(I_{n\times n}\,\,I_{n\times n})\in \mathrm{LG}(n,2n)$, which is a contradiction.

By Theorem \ref{gwondl}, each $G$-invariant closed subvariety of  $\mathcal {TL}_{n}$ has an open subset contained in the Mille Cr\^epes coordinate chart $\mathcal A_l$, for a certain $0\leq l\leq n$. On the other hand, $B_k\cap \mathcal A_l=\emptyset$ for $1\leq k\leq n-1$, by (\ref{te}) and (\ref{bstl}). Therefore, $B_k$ contains no $G$-orbit. \,\,\,$\endpf$
\medskip

In our case, there is a more geometric interpretation of the colors $B_k$ as follows. For $1\leq j\leq {n-1}$, denote by $h_j$ the pullback of the positive generator of $\operatorname{Pic}(\mathbb {P}^{N^j_{n}})$ to $\mathbb {P}^{N_{n}}\times\mathbb {P}^{N^0_{n}}\times\cdots\times\mathbb {P}^{N^n_{n}}$. Let $H_j$ be its restriction to $\mathcal {TL}_{n}$. By computing the Pl\"ucker coordinate functions $\{P_{I}\}_{I\in\mathbb I^k_{n}}$ in $\mathcal A_0$ and $\mathcal A_n$, we can conclude that $H_k=B_k$ for all $1\leq k\leq n-1$.

In what follows, we assume  $\mathbb K$ is of characteristic $0$. Recall from \cite{Br3,Br5,Per} that
\begin{theorem}\label{sph}For a smooth complete spherical variety $X$ with a Borel subgroup $B$, $\operatorname{Pic}(X) = N^1(X)$ %the nef and globally generated line bundles agree on $X$, 
and any effective cycle is rationally equivalent to an effective $B$-stable cycle.
\end{theorem}

By (\ref{bstl}) and Proposition \ref{gbl}, we can conclude that
\begin{proposition}\label{effl}
$\operatorname{Eff}(\mathcal{TL}_{n})$ is generated by $\{D_{1}^+,D_{2}^+,\cdots,D_n^+,D_{1}^-,D_{2}^-,\cdots,D_n^-\}$.   
\end{proposition}

Next, we define certain $T$-invariant curves in $\mathcal {TL}_{n}$ as the closures of affine lines in the Mille Cr\^epes coordinate charts $\mathcal A_l$. 
For $0\leq l\leq n-1$,  let $\gamma_l$ be the closure of the affine line $\mathring{\gamma}_l:\mathbb K\rightarrow \mathcal A_l$ defined by setting $b_{l+1}\left(\mathring{\gamma}_l(t)\right)=t$ and all other coordinates to zero. 
For $1\leq j\leq n-1$, let $\zeta^-_j$ (resp. $\zeta^+_j$) be the closure of $\mathring{\zeta}^-_j\subset \mathcal A_0$ (resp. $\mathring{\zeta}^+_j\subset \mathcal A_n$) defined by setting
$\xi_{j(j+1)}(\mathring{\zeta}^-_j(t))=t$ (resp. $\xi_{(n-j)(n+1-j)}(\mathring{\zeta}^+_j(t))=t$) and all the other coordinates to zero.

\begin{lemma}\label{i1l} For $0\leq l\leq n-1$, the non-zero intersection numbers of $\gamma_l$ with the generators of the Picard group are 
\begin{equation*}%\label{cone1}
\left\{
\begin{array}{ll}
   \gamma_l\cdot H=1\,\,\,\, \\
  \gamma_l\cdot D^-_{i}=\left\{ \begin{array}{cl}
       1\,\,\,\,&   i=l+1\,\,{\rm and\,\,}1\leq i\leq n-1 \\
       -1\,\,\,\,& i=l+2\,\,{\rm and\,\,}1\leq i\leq n-1 \\
  \end{array}\right.\\
  \gamma_l\cdot D^+_{i}=\left\{ \begin{array}{cl}
       1\,\,\,\,&   i=n-l\,\,{\rm and\,\,} 1\leq i\leq n-1 \\
       -1\,\,\,\,& i=n-l+1\,\, {\rm and\,\,} 1\leq i\leq n-1 \\
  \end{array}\right.   &  \\
\end{array}\right.;
\end{equation*}
for $1\leq j\leq n-1$, the non-zero intersection numbers of $\zeta^-_j$ (resp. $\zeta^+_j$) are 
\begin{equation*}%\label{cone2}
\zeta^-_j\cdot D^-_{i}\,\,({\rm resp}.\,\,\zeta^+_j\cdot D^+_{i}) =\left\{ \begin{array}{cl}
       -1\,\,\,\,&   i=j\\
       2\,\,\,\,& i=j+1\,\,{\rm and\,\,}1\leq i\leq n-1\\
       -1\,\,\,\,&   i=j+2\,\,{\rm and\,\,}1\leq i\leq n-1 \\
  \end{array}\right..
\end{equation*}
\end{lemma}
{\bf\noindent Proof of Lemma \ref{i1l}.} We view ${\gamma}_l$ as a curve in $\mathbb {P}^{N_{n}} \times\mathbb{P}^{N^0_{n}} \times\cdots\times\mathbb {P}^{N^n_{n}}$. 
By computing the Pl\"ucker coordinate functions, it is easy to verify that ${\gamma}_l\cdot H=1$ and ${\gamma}_l\cdot H_k=0$, $1\leq k\leq n-1$, for the projection of ${\gamma}_l$ to $\mathbb {P}^{N_{n}}$ is a line and that to any $\mathbb {P}^{N^j_{n}}$ is a point.
Noticing that $\gamma_l$ is disjoint with the divisors  $D_1^-,D^-_2,\cdots,D^-_l,D^+_1,D^+_2,\cdots,D^+_{n-l-1}$, we conclude Lemma \ref{i1l} by solving linear equations. The intersection numbers of $\zeta^-_j$, $\zeta^+_j$ can be computed similarly.\,\,\,\,$\endpf$
%\medskip

%For any variety $X$, denote by $Z_1(X)$ the group of one-dimensional cycles on $X$ and by $N_1(X)$ its quotient by numerical equivalence.  In $N_1(X)\otimes_{\mathbb Z}\mathbb R$, denote by $NE(X)$ the cone spanned by the classes of effective curves. We have

\begin{lemma}\label{n1}
The cone of effective curves $\operatorname{NE}(\mathcal {TL}_n)$ is generated by $\{\gamma_l\}_{0\leq l\leq n-1}\cup\{\zeta^{\pm}_j\}_{j=1}^{n-1}$.
\end{lemma}
{\bf\noindent Proof of Lemma \ref{n1}.} By Theorem \ref{sph}, it suffices to consider $T$-invariant curves. Notice that each  $T$-invariant curve has a single non-zero coordinate in the Mille Cr\^epes coordinate charts. By certain permutations, we can enumerate the generators of the classes of effective curves.  \,\,\,\endpf

\begin{proposition}\label{nefcl}
The nef cone of $\mathcal {TL}_n$ is generated by (\ref{nefll}).
\end{proposition}
{\bf\noindent Proof of Proposition \ref{nefcl}.} It is clear that $H, \Delta_{1}^+,\Delta_2^+,\cdots,\Delta_{n-1}^+, \Delta_{1}^-,\Delta_2^-,\cdots,\Delta_{n-1}^-$ is a basis of $\operatorname{Pic}(\mathcal {TL}_{n})$ (see (\ref{dell}) for  $\Delta^{\pm}_k$). By Lemma \ref{i1l}, we conclude that $\mathrm{Nef}(\mathcal {TL}_{n})$ consists of divisors 
$D:=hH+\sum\nolimits_{i=1}^{n-1}a_i\Delta_i^-+\sum\nolimits_{j=1}^{n-1}b_j\Delta_j^+$
satisfying the conditions
\begin{equation}\label{subj}
\left\{\begin{aligned}
&h-\sum\nolimits_{i=l+1}^{n-1}a_i-\sum\nolimits_{j=n-l}^{n-1}b_j\geq0,\,\,\,\,\,\,0\leq l\leq n-1\\
&a_1,\cdots,a_{n-1},b_1,\cdots,b_{n-1}\geq 0
\end{aligned}\right..\end{equation}

We prove the statement by induction on $n$. When $n=2$, it is easy to verify that the cone defined by (\ref{subj}) is generated by $(h,a_1,b_1)=(1,1,1),(1,1,0),(1,0,1),(1,0,0)$.

Now we assume that Proposition \ref{nefcl} holds for $n=m-1\geq2$ and proceed to prove it for $n=m$. Take a primitive generator of $\mathrm{Nef}(\mathcal {TL}_{m})$. If $a_k,b_{m-k}\geq 1$ for a certain $1\leq k\leq m-1$, then by (\ref{bstl}), $D-\operatorname{min}\{a_k,b_{m-k}\}\cdot B_k$ also satisfies (\ref{subj}), which is a contradiction. If $a_k,b_{m-k}=0$ for a certain $1\leq k\leq m-1$, we delete the variables $a_k,b_{m-k}$ and reorder the others so that (\ref{subj}) takes the same form with $m$ replaced by $m-1$, which thus yields Proposition \ref{nefcl} by induction.  

In the following, we assume that $D$ is a primitive generator of $\mathrm{Nef}(\mathcal {TL}_{m})$ such that for each $1\leq k\leq m-1$, either $a_k=0$ and $b_{m-k}\geq 1$, or $a_k\geq 1$ and $b_{m-k}=0$.
Note that to obtain an extremal ray, we need the intersection of at least $2m-2$ supporting hyperplanes.  Then at least $m-1$ supporting hyperplanes come from the first $m$ inequalities of (\ref{subj}). Since $m\geq 3$, there are two adjacent ones. If we subtract them, we should have $a_k=b_{m-k}=0$ for a certain $1\leq k\leq m-1$, which is a contradiction.
The proof is complete.\,\,\,\endpf

\begin{remark}\label{chp}
The above cone is precisely the cone over the chain polytope $\mathcal{C}(P)$ of the poset $P$ defined on the set $\{a_1,\dots,a_n,b_1,\dots,b_n\}$ with the following cover relations as in Figure \ref{Fig:poset}:
\[
a_1 < a_2 < \cdots < a_n,\,\,\,
b_1 < b_2 < \cdots < b_n,\,\,\,
a_i < b_j \quad \text{whenever} \quad i+j > n.
\]
The antichains of $P$ are
the empty set, singletons $\{a_i\}$ or $\{b_j\}$,
and pairs $\{a_i, b_j\}$ with $i+j \leq n$. The characteristic vectors of these antichains correspond precisely to the given vectors (identifying $a_i$ with $e_{a_i}$, $b_j$ with $e_{b_j}$, and pairs with sums). By Stanley's foundational work \cite{Sta}, the chain polytope $\mathcal{C}(P)$ is the convex hull of these characteristic vectors.
\begin{center}
\begin{figure}
\centering
\begin{tikzpicture}
[scale=0.6,
    >=stealth,
    thick,
    big/.style   = {circle,draw,fill=black,inner sep=1.5pt},
    small/.style = {circle,draw,fill=red!70,inner sep=1pt}
]

%---- main vertices -------------------------------------------------
\node[color=SkyBlue, circle, fill=SkyBlue, opacity=0.6, inner sep=8pt, label=center:$\color{white}\scriptstyle1$] (1) at (0,4) {};
\node[color=SkyBlue, circle, fill=SkyBlue, opacity=0.6, inner sep=8pt, label=center:$\color{white}\scriptstyle2$] (2) at (2,4) {};
\node[color=SkyBlue, circle, fill=SkyBlue, opacity=0.6, inner sep=8pt, label=center:$\color{white}\scriptstyle3$] (3) at (4,4) {};
\node[color=SkyBlue, circle, fill=SkyBlue, opacity=0.6, inner sep=8pt, label=center:$\color{white}\scriptstyle n-1$] (n-1) at (12,4) {};
\node[color=SkyBlue, circle, fill=SkyBlue, opacity=0.6, inner sep=8pt, label=center:$\color{white}\scriptstyle n$] (n) at (14,4) {};
\node[color=DarkRed, circle, fill=DarkRed, opacity=0.8, inner sep=8pt, label=center:$\color{white}\scriptstyle1$] (1') at (0,0) {};
\node[color=DarkRed, circle, fill=DarkRed, opacity=0.8, inner sep=8pt, label=center:$\color{white}\scriptstyle2$] (2') at (2,0) {};
\node[color=DarkRed, circle, fill=DarkRed, opacity=0.8, inner sep=8pt, label=center:$\color{white}\scriptstyle3$] (3') at (4,0) {};
\node[color=DarkRed, circle, fill=DarkRed, opacity=0.8, inner sep=8pt, label=center:$\color{white}\scriptstyle n-1$] (n-1') at (12,0) {};
\node[color=DarkRed, circle, fill=DarkRed, opacity=0.8, inner sep=8pt, label=center:$\color{white}\scriptstyle n$] (n') at (14,0) {};

%---- arrows --------------------------------------------------------
\draw[->, color=gray] (2) -- (1);
\draw[->, color=gray] (3) -- (2);
\draw[->, color=gray] (2') -- (1');
\draw[->, color=gray] (3') -- (2');
\draw[->, color=gray, color=gray] (n') -- (n-1');
\draw[->, color=gray] (n) -- (n-1);

\draw[->, color=gray] (1') -- (n);
\draw[->, color=gray] (2') -- (n);
\draw[->, color=gray] (2') -- (n-1);
\draw[->, color=gray] (3') -- (n);
\draw[->, color=gray] (3') -- (n-1);

\draw[->, color=gray] (n-1') -- (n);
\draw[->, color=gray] (n-1') -- (n-1);
\draw[->, color=gray] (n-1') -- (3);
\draw[->, color=gray] (n-1') -- (2);

\draw[->, color=gray] (n') -- (1);
\draw[->, color=gray] (n') -- (2);
\draw[->, color=gray] (n') -- (3);
\draw[->, color=gray] (n') -- (n-1);
\draw[->, color=gray] (n') -- (n);

\draw[->, color=gray] (5,4) -- (3);
\draw[->, color=gray] (5,0) -- (3');
\draw[->, color=gray] (n-1) -- (11,4);
\draw[->, color=gray] (n-1') -- (11,0);

\node at (7,4) {$\ldots$};
\node at (9,4) {$\ldots$};
\node at (7,0) {$\ldots$};
\node at (9,0) {$\ldots$};
\end{tikzpicture}
\caption{The poset $P$, $a_i$ in blue and $b_i$ in red.}
\label{Fig:poset}
\end{figure}
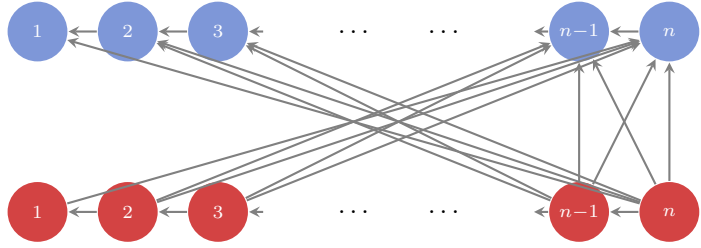
\end{center}

\end{remark}
%\begin{corollary}\label{wfl}The anticanonical bundle of $\mathcal {TL}_n$  is nef and big.  $\mathcal {TL}_n$ is Fano if and only if $n=1,2$. \end{corollary} {\bf\noindent Proof of Proposition \ref{wfl}.} It follows from Lemmas \ref{excep} and \ref{i1l}.\,\,\,\endpf

\subsection{Effective and nef cones of \texorpdfstring{$\mathcal {TO}_n$}{dd}}\label{kcpto}
Denote the canonical bundle of $\mathcal {TO}_{n}$ by $K$, and the line bundle $(RO_{n})^*\left(\mathcal O_{\mathrm{OG}^+(n,2n)}(1)\right)$ by $H$. Similarly to Lemma \ref{excep}, we can prove that
\begin{lemma}\label{excepo}
When $n$ is even, the exceptional divisor of the blow-up ${RO}_{n}$ is given by $D_{1}^++\cdots+D_{[\frac{n}{2}]-1}^++D_{1}^-+\cdots+D_{[\frac{n}{2}]-1}^-$,
$\operatorname{Pic}(\mathcal {TO}_{n})$  is freely generated by $H,D^{\pm}_1,\cdots,D^{\pm}_{[\frac{n}{2}]-1}$, and 
\begin{equation*}%\label{kano}
K=-2(n-1) H+\sum\limits_{i=1}^{[\frac{n}{2}]-1}\left(\frac{(n+3-2i)(n-2i)}{2}\right) (D_{i}^- +D_{i}^+)=-4\sum\limits_{m=1}^{[\frac{n}{2}]-1}B_m-\sum\limits_{i=1}^{[\frac{n}{2}]}(D_{i}^- +D_{i}^+).
\end{equation*}
When $n$ is odd, the exceptional divisor of ${RO}_{n}$ is $D_{1}^++\cdots+D_{[\frac{n}{2}]-1}^++D_{1}^-+\cdots+D_{[\frac{n}{2}]}^-$, $\operatorname{Pic}(\mathcal {TO}_{n})$ is freely generated by $H,D^+_1,\cdots,D^+_{[\frac{n}{2}]-1},D^-_1,\cdots,D^-_{[\frac{n}{2}]}$, and 
\begin{equation}\label{kan}
\begin{split}
K=&-2(n-1) H+\sum\limits_{i=1}^{[\frac{n}{2}]-1}\left(\frac{(n+2-2i)(n-1-2i)}{2}\right) D_{i}^+\\
&+\sum\limits_{i=1}^{[\frac{n}{2}]}\left(\frac{(n+3-2i)(n-2i)}{2}\right)D_{i}^-=-3B_{[\frac{n}{2}]}-4\sum\limits_{m=1}^{[\frac{n}{2}]-1}B_m-\sum\limits_{i=1}^{[\frac{n}{2}]}(D_{i}^- +D_{i}^+)\,.
\end{split}
\end{equation}
\end{lemma}

Denote by $B$ the Borel subgroup of ${\rm GL}_n$ consisting of  upper triangular matrices. It is clear that over $\mathbb K$, $\mathcal {TO}_{n}$ is a spherical ${\rm GL}_n$-variety. 
For $0\leq k\leq [n/2]$, define 
\begin{equation*}
b_{k}:=\left\{x\in\left. \mathrm{OG}^+(n,2n)\subset \mathbb P\big(\bigoplus\nolimits_{i=0}^{[n/2]}\bigwedge\nolimits^{2i}V^*\big) \right|\,z_{\underline I_k}(x)=0\,\right\}   
\end{equation*}
where $\underline I_k:=(1,2,\cdots,2k)$. Let $B_{k}\subset\mathcal {TO}_{n}$ be its strict transform. Similarly, for $n$ even, \begin{equation}\label{bsto}
\begin{split}
&B_k= H-\sum\limits_{i=1}^{[\frac{n}{2}]-k}([\frac{n}{2}]-k+1-i) D^+_i-\sum\limits_{i=1}^k(k+1-i)D^-_i\,\,{\rm for}\,\,1\leq k\leq [\frac{n}{2}]-1,\\
&B_0= D^+_{[\frac{n}{2}]}=H-\sum\limits_{i=1}^{[\frac{n}{2}]-1}([\frac{n}{2}]+1-i) D^+_i,\,\,\,B_{[\frac{n}{2}]}= D^-_{[\frac{n}{2}]}=H-\sum\limits_{i=1}^{[\frac{n}{2}]-1}([\frac{n}{2}]+1-i) D^-_i\,.
\end{split}
\end{equation}
For $n$ odd, we shall modify $B_{[\frac{n}{2}]}$ by $B_{[\frac{n}{2}]}= H-\sum\nolimits_{i=1}^{[\frac{n}{2}]}([\frac{n}{2}]+1-i) D^-_i$.

Similarly to Proposition \ref{gbl}, we have
\begin{proposition}\label{gbo}
The irreducible $G$-invariant divisors of $\mathcal {TO}_{n}$ are $D_{1}^{\pm},  \cdots, D_{[\frac{n}{2}]}^{\pm}$. The colors of $\mathcal {TO}_{n}$ are
$B_1,\cdots, B_{[\frac{n}{2}]-1}$ for $n$ even, and $B_1,\cdots, B_{[\frac{n}{2}]}$ for $n$ odd. Moreover, $\mathcal {TO}_{n}$ is toroidal.
\end{proposition}

In what follows, we assume $\mathbb K$ has characteristic $0$. Computation yields that
\begin{proposition}\label{effo}
When $n$ is even, $\operatorname{Eff}(\mathcal{TO}_{n})$ is generated by $\{D_{1}^{\pm},\cdots,D_{[\frac{n}{2}]}^{\pm}\}$. When $n$ is odd, $\operatorname{Eff}(\mathcal{TO}_{n})$ is generated by $\{B_{[\frac{n}{2}]},D_{1}^{\pm},D_{2}^+,\cdots,D_{[\frac{n}{2}]}^{\pm}\}$.   
\end{proposition}

For $0\leq l\leq [\frac{n}{2}]-1$, let $\gamma_l$ be the closure of ${\mathring{\gamma}}_l:\mathbb K\rightarrow \underline{\mathcal A}_l$ defined by setting $b_{l+1}\left({\mathring{\gamma}}_l(t)\right)=t$ and other coordinates to zero. For $1\leq j\leq [\frac{n-1}{2}]$ (resp. $1\leq j\leq [\frac{n}{2}]-1$), let $\zeta^-_j$ (resp. $\zeta^+_j$) be the closure of  ${\mathring{\zeta}}^-_j\subset \underline {\mathcal A}_0$ (resp. ${\mathring{\zeta}}^+_j\subset \underline {\mathcal A}_{[\frac{n}{2}]}$) defined by setting
$\xi_{(2j)(2j+1)}({\mathring{\zeta}}^-_j(t))=t$ (resp. $\xi_{(2[\frac{n}{2}]-2j+1)(2[\frac{n}{2}]-2j)}({\mathring{\zeta}}^+_j(t))=t$) and other coordinates to zero. Similarly, we can conclude 
\begin{lemma}\label{i1o} For $0\leq l\leq [\frac{n}{2}]-1$, the non-zero intersection numbers of $\gamma_l$ are 
\begin{equation*}%\label{cone1}
\left\{
\begin{array}{ll}
   \gamma_l\cdot H=1\,\,\,\, \\
  \gamma_l\cdot D^-_{i}=\left\{ \begin{array}{cl}
       1\,\,\,\,&   i=l+1\,\,{\rm and\,\,}1\leq i\leq [\frac{n-1}{2}] \\
       -1\,\,\,\,& i=l+2\,\,{\rm and\,\,}1\leq i\leq [\frac{n-1}{2}] \\
  \end{array}\right.\\
  \gamma_l\cdot D^+_{i}=\left\{ \begin{array}{cl}
       1\,\,\,\,&   i=[\frac{n}{2}]-l\,\,{\rm and\,\,} 1\leq i\leq [\frac{n}{2}]-1 \\
       -1\,\,\,\,& i=[\frac{n}{2}]-l+1\,\, {\rm and\,\,} 1\leq i\leq [\frac{n}{2}]-1 \\
  \end{array}\right.   &  \\
\end{array}\right.;
\end{equation*}
for $1\leq j\leq [\frac{n-1}{2}]$ (resp. $[\frac{n}{2}]-1$), the non-zero intersection numbers of $\zeta^-_j$ (resp. $\zeta^+_j$) are 
\begin{equation*}%\label{cone2}
\zeta^-_j\cdot D^-_{i}\,\,({\rm resp}.\,\,\zeta^+_j\cdot D^+_{i})=\left\{ \begin{array}{cl}
       -1\,\,\,\,&   i=j\\
       2\,\,\,\,& i=j+1\,\,{\rm and\,\,}1\leq i\leq [\frac{n-1}{2}]\,\,({\rm resp}.\,\,[\frac{n}{2}]-1)\\
       -1\,\,\,\,&   i=j+2\,\,{\rm and\,\,}1\leq i\leq [\frac{n-1}{2}]\,\,({\rm resp}.\,\,[\frac{n}{2}]-1) \\
  \end{array}\right..
\end{equation*}
\end{lemma}

\begin{proposition}\label{nefco}
The nef cone of $\mathcal {TO}_n$ is generated by (\ref{nefo}).
\end{proposition}
{\bf\noindent Proof of Proposition \ref{nefco}.} By Lemma \ref{i1o},  $\mathrm{Nef}(\mathcal {TO}_{n})$ consists of divisors 
$D:=hH+\sum\nolimits_{i=1}^{[\frac{n-1}{2}]}a_i\Delta_i^-+\sum\nolimits_{j=1}^{[\frac{n}{2}]-1}b_j\Delta_j^+$
satisfying the conditions
\begin{equation}\label{subjo}
\left\{\begin{aligned}
&h-\sum\nolimits_{i=l+1}^{[\frac{n-1}{2}]}a_i-\sum\nolimits_{j=[\frac{n}{2}]-l}^{[\frac{n}{2}]-1}b_j\geq0,\,\,\,\,\,\,0\leq l\leq [\frac{n}{2}]-1\\
&a_1,\cdots,a_{[\frac{n-1}{2}]},b_1,\cdots,b_{[\frac{n}{2}]-1}\geq 0
\end{aligned}\right..\end{equation}

When $n$ is even, the proof is exactly the same as that of Proposition \ref{nefcl}. We modify the induction argument for the case $n$ odd as follows. When $n=3$, the cone defined by (\ref{subjo}) is generated by $(h,a_1)=(1,1),(1,0)$. When $n=5$, the cone  is generated by $(h,a_1,a_2,b_1)=(1,1,0,1),(1,1,0,0),(1,0,1,0),(1,0,0,1),(1,0,0,0)$. Assume that Proposition \ref{nefcl} holds for $n=2m-1\geq5$, and we proceed to prove it for $n=2m+1$.

Take a primitive generator of $\mathrm{Nef}(\mathcal {TO}_{2m+1})$. If $a_k,b_{m-k}\geq 1$ for a certain $1\leq k\leq m-1$, then by (\ref{bsto}), $D-\operatorname{min}\{a_k,b_{m-k}\}\cdot B_k$ also satisfies (\ref{subjo}). If $a_k=b_{m-k}=0$ for a certain $1\leq k\leq m-1$, we delete $a_k,b_{m-k}$ and reorder the others so that (\ref{subjo}) takes the same form with $n$ replaced by $n-2$, which thus yields Proposition \ref{nefcl} by induction. 

We claim that there must be a certain $1\leq k\leq m-1$ such that $a_k=b_{m-k}=0$, or $a_k,b_{m-k}\geq 1$. Otherwise, among the supporting hyperplanes that cut out a given extremal ray, at least $m-1$ of them arise from the first $m$ inequalities in (\ref{subjo}). Since $m\geq 3$, there are two adjacent ones. If we subtract them, we should have $a_k=b_{m-k}=0$ for a certain $1\leq k\leq m-1$, which is a contradiction.
The proof is complete.\,\,\,\endpf
\medskip

{\bf\noindent Proof of Theorem \ref{fano}.} It follows from Lemmas \ref{excepo}, \ref{excep}, \ref{i1l}, \ref{i1o}.\,\,\,\endpf
\medskip

{\bf\noindent Proof of Theorem \ref{bignef}.}
The proof is similar to that of  \cite[Theorem 1.2]{FZ21}. Since $\mathcal {TL}_{n}$, $\mathcal {TO}_n$ are toroidal, Theorem \ref{bignef} holds when they are Fano (see \cite[Corollary 3.3.11]{Per}). 

To avoid repetition, we present the proof for $\mathcal {TO}_n$ only. Without loss of generality, we further assume that $n=2m+1$ is odd. We have the following exact sequence by \cite[Proposition 2.3.2]{BB}.
\begin{equation*}%\label{exa}
0\rightarrow\mathcal{S}_{\mathcal {TO}_{n}}\rightarrow \mathcal{T}_{\mathcal {TO}_{n}}\rightarrow\bigoplus\nolimits_{i=1}^{m}\mathcal O_{\mathcal {TO}_{n}}(D_i^-)|_{D_i^-}\bigoplus\nolimits_{i=1}^{m}\mathcal O_{\mathcal {TO}_{n}}(D_i^+)|_{D_i^+}\rightarrow 0,
\end{equation*}
where $\mathcal S_{\mathcal {TO}_{n}}\subset\mathcal T_{\mathcal {TO}_{n}}$ consists of vector fields tangent to the boundary. Notice that for $k>0$, $H^k(\mathcal {TO}_{n}, \mathcal S_{\mathcal {TO}_{n}})=0$ (see \cite[Theorem 3.2]{Per}), and $\mathcal O_{\mathcal {TO}_{n}}(D_i^{\pm})_{D_i^{\pm}}\cong-K_{\mathcal {TO}_{n}}|_{D^{\pm}_i}+K_{D^{\pm}_i}$. 

{\bf Claim.} The restriction $-K_{\mathcal {TO}_{n}}|_{D^{\pm}_i}$ is nef and big.

{\bf Proof of Claim.} 
It suffices to show that $-K_{\mathcal {TO}_{n}}|_{D^{\pm}_i}$ is in the interior of $\operatorname{Eff}(\mathcal{TO}_{n})$. Without loss of generality, we consider the restriction of (\ref{kan}) to $D^-_j$ with $1\leq j\leq [\frac{n}{2}]$ and $n\geq 3$. 

Denote by $B^{-j}_{m}$, $0\leq m\leq [\frac{n}{2}]$, the restriction of the line bundle $B_m$ to $D^{-}_j$, and by $D^{-j}_{\pm i}$, $1\leq i\leq [\frac{n}{2}]$, the restriction of $D^{\pm}_i$ to $D^-_j$. %Note that $D^{-j}_{-i}$ can be identified with the intersection $D^{-}_j\cap D^{-}_i$, when $i\neq j$; $D^{-j}_{+i}$ can be identified with $D^{-}_j\cap D^{+}_i$, which is empty when $1\leq i\leq [\frac{n}{2}]+1-j$. %Similar results hold for $D^-_j$. 
It is easy to verify that $D^{-}_j$ is a spherical variety with boundary divisors $\{D^{-j}_{-i}\}_{1\leq i\leq [\frac{n}{2}], i\neq j}\cup\{D^{-j}_{+i}\}_{[\frac{n}{2}]+2-j\leq i\leq [\frac{n}{2}]}$ and colors $\{B^{-j}_m\}_{1\leq m\leq [\frac{n}{2}]-1}$. Moreover, its effective cone is generated by $\{D^{-j}_{-i}\}_{1\leq i\leq [\frac{n}{2}], i\neq j}\cup\{D^{-j}_{+i}\}_{[\frac{n}{2}]+2-j\leq i\leq [\frac{n}{2}]}\cup\{B^{-j}_{j-1}\}$, where $B^{-1}_0 = \emptyset$ if $j = 1$.  %we have that\begin{equation*} D^{-j}_{-j}\sim\left\{ \begin{array}{ll} B^{-j}_0-B^{-j}_1\,\,\,\,&   {\rm when}\,\,j=1\\  2B^{-j}_{1}-B^{-j}_{2}-B^{-j}_{0}\,\,\,\,&  {\rm when}\,\,j=2\\  
% 2B^{-j}_{j-1}-B^{-j}_{j}-B^{-j}_{j-2}-D^{-j}_{+(n+2-j)}\,\,\,\,&  {\rm when}\,\,3\leq j\leq n-1\\ 2B^{-j}_{j-1}-B^{-j}_{j-2}\,\,\,\,& {\rm when}\,\,   j=n\\  \end{array}\right..\end{equation*}

When $3\leq j\leq [\frac{n}{2}]-1$, we have that
\begin{equation*}
\sum\nolimits_{1\leq i\leq j}D^{-j}_{-i}=B^{-j}_{j-1}-B^{-j}_j\,\,\,\,\,\,\,\,{\rm and}\,\,\,\,\,\,\sum\nolimits_{1\leq i\leq j-1}D^{-j}_{-i}=B^{-j}_{j-2}-B^{-j}_{j-1}+D^{-j}_{+([\frac{n}{2}]+2-j)}.
\end{equation*}
Setting $0<t<1$, we get
\begin{equation*}
\begin{split}
-K_{\mathcal {TO}_{n}}|_{D^{-}_j}&= t\cdot\sum\nolimits_{1\leq i\leq j-1}D^{-j}_{-i}+
\sum\nolimits_{j+1\leq i\leq [\frac{n}{2}]} D^{-j}_{-i}+\sum\nolimits_{[\frac{n}{2}]+3-j\leq i\leq [\frac{n}{2}]}D^{-j}_{+i}+(1-t)\cdot D^{-j}_{+([\frac{n}{2}]+2-j)}\\
&\,\,\,\,\,\,\,\,\,+(2+t)\cdot B_{j-1}^{-j}+(1-t)\cdot B_{j-2}^{-j}+\sum\nolimits_{1\leq m\leq [\frac{n}{2}]-1,\,m\neq j-2,j-1,j}B_{m}^{-j}+3B^{-j}_{[\frac{n}{2}]}.
\end{split} 
\end{equation*}

The same argument works for $j=1,2,n$. The proof is complete. \,\,\,\endpf

Now by the Kawamata-Viehweg vanishing theorem,  $H^k(\mathcal {TO}_{n}, T_{\mathcal {TO}_{n}})=0$ for $k>0$. The similar argument works for $\mathcal {TO}_n$. We complete the proof of Theorem \ref{bignef}.
\,\,\,\,$\endpf$

\section*{Acknowledgment}

H. Fang and X. Wu are supported by National Key R\&D Program of China (No.~2022YFA1006700) and NSFC grant (No.~12201012). Alex Massarenti was supported by the PRIN 2022 grant (project 20223B5S8L) \emph{Birational geometry of moduli spaces and special varieties}, is a member of GNSAGA (INdAM), and thanks Peking University for the hospitality during a visiting period in which the collaboration leading to this article began.

\bibliographystyle{plain} 
\bibliography{ref_260228}

\end{document}